\documentclass[a4paper,12pt]{amsart}
\usepackage{amssymb,amsmath,latexsym,color,graphicx}
\usepackage{hyperref}
\usepackage{amsfonts}
\usepackage{amsthm}
\usepackage{tikz}
\usetikzlibrary{positioning}

\newtheorem{dref}{Definition}[section] 
\newtheorem{theo}[dref]{Theorem} 
\newtheorem{prop}[dref]{Proposition}
\newtheorem{remark}[dref]{Remark} 

\newtheorem{cor}[dref]{Corollary}

\def\pmtwo#1#2#3#4{\left( \begin{array}{cc}#1&#2\\#3&#4\end{array}\right)}
\def\dbar{\overline{\partial}}

\begin{document}

\title[]{Large $|k|$ behavior for the  reflection coefficient for 
Davey-Stewartson II equations}
\author{Christian Klein}
\address{Institut de Math\'ematiques de Bourgogne, UMR 5584\\
                Universit\'e de Bourgogne-Franche-Comt\'e, 9 avenue Alain Savary, 21078 Dijon
                Cedex, France\\
				Institut Universitaire de France\\
    E-mail Christian.Klein@u-bourgogne.fr}
    
\author{Johannes Sj\"ostrand}
\address{Institut de Math\'ematiques de Bourgogne, UMR 5584\\
                Universit\'e de Bourgogne-Franche-Comt\'e, 9 avenue Alain Savary, 21078 Dijon
                Cedex, France\\
    E-mail Johannes.Sjostrand@u-bourgogne.fr}
    
\author{Nikola Stoilov}
\address{Institut de Math\'ematiques de Bourgogne, UMR 5584\\
                Universit\'e de Bourgogne-Franche-Comt\'e, 9 avenue Alain Savary, 21078 Dijon
                Cedex, France \\ 
                and  
Laboratoire Jaques-Louis Lions, UMR 7598  \\
Faculte des Sciences er Ingeni\'ere, Sorbonne Universit\'e, 4 Place Jussieu 75005 Paris, France \\        
    E-mail Nikola.Stoilov@ljll.math.upmc.fr}
\date{\today}
\begin{abstract}
The study of complex geometric optics solutions to a system of d-bar equations appearing
in the context of electrical impedance tomography and the scattering theory of
the integrable Davey-Stewartson II equations for large values of the
spectral parameter $k$ in \cite{KlSjSt20} is extended to the 
reflection coefficient. For the case 
of potentials $q$ with compact support on some domain $\Omega$ with 
smooth strictly convex boundary, improved asymptotic relations are 
provided.  
\end{abstract}

\thanks{This work is partially supported by 
the ANR-FWF project ANuI - ANR-17-CE40-0035, the isite BFC project 
NAANoD, the EIPHI Graduate School (contract ANR-17-EURE-0002) and by the 
European Union Horizon 2020 research and innovation program under the 
Marie Sklodowska-Curie RISE 2017 grant agreement no. 778010 IPaDEGAN}

\maketitle
\tableofcontents

\section{Introduction}
\setcounter{equation}{0}

This paper  addresses the scattering problem for the 
integrable Davey-Stewartson (DS) II equation given by the Dirac system 
\begin{equation}\label{dbarphi}
  \begin{cases}
    \overline{\partial}\phi_{1}=\frac{1}{2}q\mathrm{e}^{\overline{kz}-kz}\phi_{2},\\
    \partial\phi_{2}=\sigma\frac{1}{2}\overline{q}\mathrm{e}^{kz-\overline{kz}}\phi_{1},\quad \sigma=\pm1,
\end{cases}   
\end{equation}
subject to the asymptotic conditions 
\begin{equation}
    \lim_{|z|\to\infty}\phi_{1}=1,\quad \lim_{|z|\to\infty}\phi_{2}=0,
    \label{Phisasym}
\end{equation}
where  $q=q(x,y)$ is a complex-valued field, where the \emph{spectral 
parameter} $k\in\mathbb{C}$ is independent of $z=x+\mathrm{i} y$, 
$(x,y)\in \mathbb{R}^{2}$, and 
where 
\begin{equation*}
\partial:=\frac{1}{2}\left(\frac{\partial}{\partial x}-\mathrm{i}\frac{\partial}{\partial y}\right)\quad\text{and}\quad
\bar{\partial}:=\frac{1}{2}\left(\frac{\partial}{\partial 
x}+\mathrm{i}\frac{\partial}{\partial y}\right).
\end{equation*}
The functions $\phi_{i}(z;k)$, $i=1,2$  depend on $z$ and $k$, and are 
called \emph{complex 
geometric optics} (CGO) solutions. Note 
that  they need not be holomorphic in either variable.

In addition to the DS II system, the CGO solutions
appear in the scattering theory of two-dimensional integrable 
equations as the Kadomtsev-Petviashvili and the Novikov-Veselov 
equation, see \cite{KSbook} for references, in 
\emph{electrical impedance tomography} (EIT), see \cite{Uhl,MS}, and Normal 
Matrix Models in Random Matrix Theory, see e.g.\ \cite{KM}.  
Our main interest in this paper is in the scattering data, the so-called  \emph{reflection 
coefficient} $R$ defined by 
\begin{equation}\label{reflc}
    \overline{R} = 
    \frac{2\sigma}{\pi}\int_{\mathbb{C}}^{}\mathrm{e}^{kz-\overline{kz}}\overline{q}(z)\phi_{1}(z;k)L(\mathrm{d}z),
\end{equation}
where $L(dz)$ is the Lebesgue measure on the complex plane. 

We are in 
particular interested in the case that the potential $q$ has compact support on some 
simply connected domain $\Omega\subset\mathbb{C}$ with a smooth boundary. 
This is a typical situation in EIT since the body of a patient has 
obviously compact support. As an example for such a situation, we 
study in this paper the case that $q$ is the characteristic function 
of the domain $\Omega$. In the context of DS II such a setting 
would correspond to a situation as in the seminal work by 
Gurevich and Pitaevski \cite{GP} for the Korteweg-de Vries (KdV) equation. 
For dispersive PDEs as DS and KdV, rapid modulated oscillations are 
expected in the 
vicinity of a discontinuity of the initial data $q$ called
\emph{dispersive shock waves} (DSW). A detailed study of this case for DS 
would allow more insight into the formation of DSWs for the DS II 
system. 

Since it is analytically difficult to solve 
d-bar systems explicitly, a large number of numerical approaches has been 
developed.  The most popular ones are based on discrete Fourier transforms 
applied to the solution of d-bar equations in terms of the
solid Cauchy transform, see  \cite{KnMuSi2004,MS}. 
The first approach along these lines with  an exponential decrease of the numerical error with the 
number of Fourier modes has been given in \cite{KM,KMS} for Schwartz  class potentials. 
A similar dependence of the error on the numerical resolutions could 
be achieved for potentials with compact support on a 
disk in \cite{KS19}. However, it is only possible to reach 
\emph{machine precision} (here $10^{-16}$) for values of $|k|\simeq 
1000$. Therefore the numerical approach \cite{KS19} was complemented in 
\cite{KlSjSt20} with explicit asymptotic formulae for large values of 
$|k|$  for potentials being the characteristic function of a compact 
domain. These results will be extended in the present paper to allow 
for sharper asymptotic results and explicit expressions for the 
reflection coefficient. 

\subsection{State of the art}
We briefly summarize the state of the art of the theory of the Dirac system 
(\ref{dbarphi}) and the results of 
\cite{KlSjSt20,KlSjSt20b} to be generalized in this paper. 
The question of existence and uniqueness of CGO solutions to system 
(\ref{dbarphi}) with $\sigma=1$ was studied in \cite{BC} for Schwartz class 
potentials and in \cite{Su1,Su2,Su3} for potentials  
$q\in L^{\infty}(\mathbb{C})\cap L^{1}(\mathbb{C})$ such that 
also $\hat{q}\in L^{\infty}(\mathbb{C})\cap L^{1}(\mathbb{C})$ 
where $\hat{q}$ is the Fourier transform of $q$ (the potentials have 
to satisfy a smallness condition in the focusing case $\sigma=-1$).
In \cite{BU} this was generalized respectively to real-valued, compactly supported potentials 
in $L^{p}(\mathbb{C})$ and in \cite{Perry2012} to potentials in 
$H^{1,1}(\mathbb{C})$ , and in \cite{NRT} to potentials 
in  $L^{2}(\mathbb{C})$.

One application of the system (\ref{dbarphi}) as shown in 
\cite{Fok,FA} is that it gives both the scattering and 
inverse scattering map for the  Davey-Stewartson II 
equation
\begin{equation}
\begin{split}
i q_t + (q_{xx}-q_{yy}) + 2\sigma(\Phi + |q|^2)q&=0,\\
\Phi_{xx}+\Phi_{yy} +2(|q|^2)_{xx}&=0,
\end{split}
    \label{eq:DSII}
\end{equation}
a two-dimensional nonlinear Schr\"odinger equation; DS is \emph{defocusing} 
for  $\sigma=1$, and \emph{focusing} for $\sigma=-1$. Note that DS systems 
appear in the modulational regime of many dispersive 
equations as for instance the water wave systems, see e.g.,  
\cite{KSbook} 
for a review on DS equations and a 
comprehensive list of references, and are only integrable for the 
choice of parameters in (\ref{eq:DSII}). 

The scattering data  are given in terms of the reflection coefficient $R=R(k)$ 
in (\ref{reflc}). The DS II equations (\ref{eq:DSII}) are completely 
integrable in the sense that a Lax pair exists, the first part of the 
Lax pair being (\ref{dbarphi}). Here $\phi_{1}$, $\phi_{2}$, $q$ can 
be seen as having a dependence on the physical time $t$ which is 
suppressed since it will not be studied in this paper. However, it 
will play a role in the second equation of the Lax pair. We put 
$\Theta=
\begin{pmatrix}
	\phi_{1}e^{kz} \\
	\phi_{2}e^{\overline{kz}}
\end{pmatrix}
$ and get for the Lax pair 
\begin{eqnarray}
\label{eq:specprob}
\Theta_{x} + i \sigma_{3} \Theta_{y} &=& \pmtwo{0}{q}{\overline{q}}{0} \Theta, \\
~\nonumber \\
\label{eq:TimeEvolve}
\Theta_{t} &=& \pmtwo
{i \dbar^{-1} \partial \left( |q|^{2} \right)/2}{-i \partial q}
{i \dbar \overline{q}} {-i \dbar^{-1} \partial \left( |q|^{2} 
\right)/2} \Theta \\
\nonumber && - \pmtwo{0}{q}{\overline{q}}{0} \Theta_{y} + i \sigma_{3}\Theta_{yy} \ ,
\end{eqnarray}
where $\sigma_{3}=\mbox{diag}(1,-1)$ is a Pauli matrix.

As $q$ in (\ref{eq:DSII}) evolves in time $t$, the reflection 
coefficient evolves  because of (\ref{eq:TimeEvolve}) by a trivial phase factor:
\begin{equation}
R(k;t)=R(k,0)e^{4i t\Re(k^2)}.
\end{equation}
The inverse scattering transform for DS II is then given by 
(\ref{dbarphi}) and
(\ref{Phisasym}) after replacing $q$ by $R$ and vice versa, the derivatives with 
respect to $z$ by the corresponding derivatives with respect to $k$, 
and asymptotic conditions for $k\to\infty$ instead of $z\to\infty$, 
see \cite{AF}.

The main interest in \cite{KlSjSt20,KlSjSt20b} and the present paper 
is in the case when $|k|$ is large, i.e., \( h 
:=1/|k|\) small. We 
introduce the following notation: 
\begin{equation*}
kz-\overline{kz}=i|k|\Re (z\overline{\omega })=i|k|\langle z,\omega
\rangle_{\mathbb{R}^2},\ \ \omega= \frac{2i\overline{k}}{|k|},
  \end{equation*}
  and for $u\in L^{2}(\mathbb{R}^{2})$
  \begin{equation*}
\widehat{\tau}_{\omega}u = e^{kz-\overline{kz}}u, \quad \widehat{\tau}_{-\omega}u = e^{\overline{kz} - kz}u, \quad E = (h\overline{\partial})^{-1}, F = (h{\partial})^{-1}
\end{equation*}
which leads for (\ref{dbarphi}) to
\begin{equation*}\label{sy.8}
\begin{cases}
h\overline{\partial }\phi
_1=\widehat{\tau }_{-\omega }h(q/2)\phi _2,\\
h\partial \phi
_2=\sigma\widehat{\tau }_\omega h(\overline{q}/2)\phi _1.
\end{cases}
\end{equation*}

See \cite{KlSjSt20}, Section 2, for the precise choice of $E$, $F$. Looking for solutions of the form (\ref{decomp.1}) - (\ref{decomp.2}) below, leads to the inhomogeneous system 
\begin{equation*}
\begin{cases}
\phi^{1} _1-E\widehat{\tau }_{-\omega }\frac{hq}{2}\phi^{1} _2=Ef_1,\\
\phi^{1} _2-\sigma F\widehat{\tau }_\omega 
\frac{h\overline{q}}{2}\phi^{1} 
_1=Ff_2,
\end{cases}
\end{equation*}
with $f_{1}= 0$ and $f_{2}=\sigma\widehat{\tau}_\omega h\bar{q}/2$, 
or 
\begin{equation*}
(1-\mathcal{K})\begin{pmatrix}\phi^{1} _1\\\phi^{1}
  _2\end{pmatrix}=\begin{pmatrix}Ef _1\\Ff _2\end{pmatrix},
\end{equation*}
where
\begin{equation}\label{sy.16}
\mathcal{K}=\begin{pmatrix}0 &A\\ B &0\end{pmatrix},\ \ 
\begin{cases}A=E\widehat{\tau }_{-\omega }\frac{hq}{2},\\
B=\sigma F\widehat{\tau }_\omega \frac{h\overline{q}}{2}\ (=\sigma\overline{A})
\end{cases} .
\end{equation}

In (\cite{KlSjSt20}) we showed, 
\begin{prop}\label{sy1}
Let $q\in \langle \cdot \rangle^{-2}H^s$ for some $s\in ]1,2]$ and fix
$\epsilon \in ]0,1]$. Define $\mathcal{ K}$ as in (\ref{sy.16}). Then
$\mathcal{ K}=\mathcal{ O}(1): (\langle \cdot \rangle^\epsilon  L^2)^2\to
(\langle \cdot \rangle^\epsilon  L^2)^2 $,
$$
\mathcal{ K}^2=\mathcal{ O}(h^{s-1}):\, (\langle \cdot \rangle^\epsilon
L^2)^2\to (\langle \cdot \rangle^\epsilon  L^2)^2.
$$
For $h_0>0$ small enough and $0<h\le h_0$,
$1-\mathcal{ K}: (\langle \cdot \rangle^\epsilon  L^2)^2\to (\langle \cdot \rangle^\epsilon  L^2)^2$
has a uniformly bounded inverse,
\begin{equation}\label{sy.32.5}
  (1-\mathcal{ K})^{-1}=(1-\mathcal{ K}^2)^{-1}(1+\mathcal{ K})=
  \begin{pmatrix} (1-AB)^{-1} &0\\ 0
    &(1-BA)^{-1}\end{pmatrix} \begin{pmatrix}
1 &A\\ B &1
  \end{pmatrix}.
\end{equation}
When
$q$ is the characteristic function of  a
  bounded strictly convex domain with smooth boundary, the conclusions
  hold with $s = \frac{3}{2}$.
\end{prop}

\subsection{Main results}

The main goal of this paper is to obtain improved asymptotics  for the 
reflection coefficient, in particular for the case of potentials with 
compact support. To this end we solve the system (\ref{dbarphi}) for $q\in \langle 
\cdot \rangle^{-2}H^s$ for some $s\in ]1,2]$ for small 
\( h\) in the form
\begin{equation}\label{decomp.1}
\phi _1 = \phi _1^0 + \phi_1^1,\quad \phi _2 = \phi _2^0 + 
\phi_2^1 .
\end{equation}
We start with
\begin{equation}\label{decomp.2}
\phi _1^0=1,\ \phi _2^0=0.
\end{equation}


The functions $\phi_{1}^{1}$ and $\phi_{2}^{1}$ should satisfy
\begin{equation*}
\begin{cases}
h\overline{\partial }\phi _1^{1}-\widehat{\tau }_{-\omega
}\frac{hq}{2}\phi _2^{1}=f _1,\\
h\partial \phi _2^{1}-\sigma\widehat{\tau }_{\omega
}\frac{h\overline{q}}{2}\phi _1^{1}=f _2,
\end{cases}
\end{equation*}
with $f_{1}=0$ and $f_{2}=\sigma \widehat{\tau}_{\omega}h\overline{q}/2$ which 
is $\mathcal{O}(h)$ in $\langle \cdot \rangle^{-2}L^2$. We look for 
$\phi _j^{1}\in
\langle \cdot \rangle^\epsilon L^2$ for $\epsilon\in
]0,1]$. 
This is equivalent to
\begin{equation*}
\begin{cases}
\phi _1^{1}-E\widehat{\tau }_{-\omega }\frac{hq}{2}\phi _2^{1}=Ef _1,\\
\phi _2^{1}-\sigma F\widehat{\tau }_\omega 
\frac{h\overline{q}}{2}\phi _1^{1}=Ff _2,
\end{cases}
\end{equation*}
or 
\begin{equation*}
(1-\mathcal{K})\begin{pmatrix}\phi _1^{1}\\\phi
  _2^{1}\end{pmatrix}=\begin{pmatrix}Ef _1\\Ff _2\end{pmatrix}.
\end{equation*}

Here $Ff_2=\sigma F\widehat{\tau }_\omega h\overline{q}/2=\mathcal{O}(1)$ in $\langle
\cdot \rangle^\epsilon L^2$ and Proposition \ref{sy1} gives us a
unique solution in $(\langle \cdot \rangle^\epsilon L^2)^2$ which is
$\mathcal{O}(1)$ in that space. 
More precisely by (\ref{sy.32.5}), we 
get 
\begin{theo}
The system (\ref{sy.8}) has the solution (\ref{decomp.1}), (\ref{decomp.2}) with
\begin{equation}%
\label{sy.41}
\phi _1^1=(1-AB)^{-1}A\sigma F\widehat{\tau }_w
\frac{h\overline{q}}{2}= (1-AB)^{-1}AB(1)
\end{equation}
\begin{equation*}
\begin{split}
\phi _2^1&=(1-BA)^{-1}\sigma F\widehat{\tau }_\omega
\frac{h\overline{q}}{2}
\end{split}
\end{equation*}
where $\sigma F\widehat{\tau }_\omega
  h\overline{q}/2=B(1)$ and $A\sigma F\widehat{\tau }_\omega
  h\overline{q}/2=AB(1)$ are $\mathcal{ O}(1)$  in
$\langle \cdot \rangle^\epsilon L^2 $ by (\ref{sy.16}). 
Note that formally $\phi_{1}=(1-AB)^{-1}(1)$, $\phi_{2}=\phi_{2}^{1}$. 
\end{theo}

On the way we improve the results of Prop.~\ref{sy1} 
for potentials $q$ being the characteristic function of a 
compact domain,
\begin{prop}\label{sy1new}
If
$q$ is the characteristic function of  a
  bounded strictly convex domain with smooth boundary, the 
  conclusions of Prop.~\ref{sy1}  
  hold with $s = 2$.
\end{prop}

In the following we will  put $\sigma=1$ for the ease of 
presentation. When $\Omega\in \mathbb{C}$ is a bounded domain, let 
\begin{equation}
	D_{\Omega}(z)=\frac{1}{\pi}\int_{\Omega}^{}\frac{1}{z-w}L(dw),
	\quad z\in\mathbb{C}
	\label{Do}
\end{equation}
be the solution of the d-bar problem
\begin{equation}
	\begin{cases}
		\partial_{\bar{z}}D_{\Omega}=1_{\Omega}, &  \\
		D_{\Omega}(z)\to 0, &  z\to \infty.
	\end{cases}
	\label{Do2}
\end{equation}
The main theorem of this paper reads 
\begin{theo}\label{maintheorem}
	Let $\Omega\Subset \mathbb{C}$ be open with a strictly convex 
	smooth boundary, and let $iu(w,k)$ be a holomorphic extension of 
	$kw -\overline{kw}$ from $\partial\Omega$ to $\mbox{neigh}(\partial \Omega,\mathbb{C})$.
	$D_{\Omega}$ is continuous, $D_{\Omega}|_{\Omega}\in 
	C^{\infty}(\overline{\Omega})$, 
	$D_{\Omega}|_{\mathbb{C}\backslash \Omega}\in C^{\infty}(
	\mathbb{C}\backslash \Omega)$. Moreover
	\begin{equation}
		\begin{split}
		\overline{R}&=\frac{2}{\pi}\int_{\Omega}^{}e^{kz-\overline{kz}} L(dz)\\
	&+\frac{1}{4i\pi |k|^{2}}\left(-\int_{\tilde{\Gamma}}^{}
	D_{\Omega}(w)e^{iu(w,k)}dw+\overline{\int_{\Gamma}^{}
	D_{\Omega}(w)e^{-iu(w,k)}dw}\right)+\mathcal{O}(|k|^{-3}\ln |k|),			
		\end{split}	
		\label{eqmt}
	\end{equation}
	for $k\in \mathbb{C}$, $|k|\gg1$. 
Here $\Gamma$ is the contour from Fig.~\ref{contour}, and 
$\tilde{\Gamma}$ is defined as $\Gamma$ after replacing $k$ with 
$-k$. 	
\end{theo}

We parametrize the positively oriented boundary $\partial\Omega$ by $\gamma(t)$, $t\in [0,L]$ and put 
$\varphi(t)=-(k\gamma(t)-\overline{k\gamma(t)})/|k|$ as well as 
$a(t)=\dot{\gamma}(t)$.
We denote the critical points of $\varphi$ by $t_{\pm}$. Applying a stationary phase approximation to reflection coefficient 
(\ref{eqmt}), we obtain 
\begin{cor}
The leading order of the reflection coefficient for $|k|\gg1$ is given by	
	\begin{equation}
		\begin{split}
		\overline{R}&=\frac{i\sqrt{2\pi}}{\pi 
		\bar{k}}\sum_{t=t_{-},t_{+}}^{}e^{-|k|\varphi(t)}\left[
		|k|^{-\frac{1}{2}}a(t)\varphi''(t)^{-\frac{1}{2}}
		+|k|^{-\frac{3}{2}}\left( 
		\frac{a''(t)}{2}\varphi''(t)^{-\frac{3}{2}}\right.\right.\\
		&\left.\left.- 
 \left( a(t)\frac{\varphi^{(4)}(t)}{6} + 
a'(t)\frac{\varphi^{(3)}(t)}{2} \right)\varphi''(t)^{-\frac{5}{2}} + 
\frac{5}{24}a(t)\varphi^{(3)}(t)^2\varphi''(t)^{-\frac{7}{2}} 
\right)\right] \\
	&+\frac{\sqrt{2\pi}}{4\pi i
	|k|^{2}}\sum_{t=t_{-},t_{+}}^{}e^{-|k|\varphi(t)}|k|^{-\frac{1}{2}}\varphi''(t)^{-\frac{1}{2}}\left(-D_{\Omega}(\gamma(t))a(t)
	+\overline{D_{\Omega}(\gamma(t))\bar{a}(t)}\right) \\
	&+\mathcal{O}(|k|^{-3}\ln |k|)	.
		\end{split}	
		\label{eqmt2}
	\end{equation}
	The branches of the square roots are chosen as in Remark \ref{rem2}.
\end{cor}

The paper is organized as follows: in section 2, we give estimates 
for the operator $AB$. In section 3 these estimates are applied to 
the reflection coefficient. In section 4 we provide explicit formulae 
via a stationary phase approximation. We consider the example of 
the characteristic function of the unit disk and give a partial proof of a 
conjecture in \cite{KlSjSt20} for the reflection coefficient in this 
case. We add some concluding remarks in section 5. 

\section{Estimates for the operator $AB$}\label{AB}
\setcounter{equation}{0}
Let $\Omega\Subset \mathbb{C}$ be strictly convex with smooth 
boundary. 
The central problem is to study $AB$ where  $A$, $B$ are
  given in (\ref{sy.16}),
\begin{equation*}
A = E \widehat{\tau}_{-\omega}\frac{hq}{2}, \quad B = \sigma F 
\widehat{\tau}_\omega\frac{h\overline{q}}{2} ( = \sigma 
\overline{A}), \quad q = 1_{\Omega},
\end{equation*}
\begin{equation*}
Au(z) = \frac{1}{2\pi}\int_{\Omega}\frac{1}{z-w}e^{-kw +\overline{kw}}u(w)L(\mathrm{d}w),
\end{equation*}
\begin{equation*}
Bu(z) = \frac{\sigma}{2\pi}\int_{\Omega}\frac{1}{\overline{z}-\overline{w}}e^{kw -\overline{kw}}u(w)L(\mathrm{d}w),
\end{equation*}
\begin{equation}
\begin{split}
ABu(z)  &= \frac{\sigma}{4\pi^2}\iint_{\Omega\times\Omega}\frac{1}{z-\zeta}e^{-k\zeta +\overline{k\zeta}}\frac{1}{\overline{\zeta} - \overline{w}}e^{kw -\overline{kw}}u(w)L(\mathrm{d}\zeta)L(\mathrm{d}w)\\
&= \int_{\Omega}K(z,w)u(w)L(\mathrm{d}w),
\end{split}
\label{0p3}
\end{equation}
\begin{equation}
  K(z,w) = \frac{\sigma}{4\pi^2}\iint_{\Omega}\frac{e^{\overline{k\zeta}}}{(z-\zeta)}\frac{e^{-k\zeta}}{(\overline{\zeta} - \overline{w})}\frac{\mathrm{d}\overline{\zeta}\wedge\mathrm{d}\zeta}{2\mathrm{i}}e^{kw -\overline{kw}},
\label{1p3}
\end{equation}
for the case of the characteristic function of $\Omega$.

We  look for functions $\widetilde{f}$, $\widetilde{g}$
  such that with $d$ denoting exterior differentiation with respect 
  to $\zeta$,
\begin{equation*}
  \frac{e^{\overline{k\zeta}}}{(z-\zeta)}\frac{e^{-k\zeta}}{(\overline{\zeta}
    - \overline{w})}\mathrm{d}\overline{\zeta}\wedge\mathrm{d}\zeta =
  \mathrm{d}\left(e^{\overline{k\zeta}-k\zeta
    }(\widetilde{f}\mathrm{d}\zeta +
    \widetilde{g}\overline{\mathrm{d}\zeta})\right)+ \dots ,
\end{equation*}
i.e.\ 
\begin{equation*}
  \frac{e^{\overline{k\zeta}}}{(z-\zeta)}\frac{e^{-k\zeta}}{(\overline{\zeta}
    - \overline{w})} =
  \partial_{\overline{\zeta}}(e^{\overline{k\zeta}-k\zeta
  }\widetilde{f})  - \partial_{{\zeta}}(e^{\overline{k\zeta}-k\zeta
  }\widetilde{g})+ \dots \ .
\end{equation*}
With
\begin{equation*}
\widetilde{f} = \frac{f}{\overline{k}(z-\zeta)(\overline{\zeta} - \overline{w})}, \quad \widetilde{g} = \frac{g}{{k}(z-\zeta)(\overline{\zeta} - \overline{w})},
\end{equation*}
we get
\begin{multline*}\partial_{\overline{\zeta}}(e^{\overline{k\zeta}-k\zeta }\widetilde{f})=
\frac{1}{\overline{k}}\partial_{\overline{\zeta}}\left(\frac{e^{\overline{k\zeta}-k\zeta} f}{(z-\zeta)(\overline{\zeta} - \overline{w})}   \right) =\\ 
\frac{e^{\overline{k\zeta}-k\zeta}}{(z-\zeta)(\overline{\zeta} - \overline{w})}
\left(1 - \frac{1}{\overline{k}(\overline{\zeta} - \overline{w})} +
\frac{1}{\overline{k}}\partial_{\overline{\zeta}}\right)f -
\frac{1}{\overline{k}}\frac{e^{\overline{k\zeta}-k\zeta}}{(\overline{\zeta}-\overline{w})}\pi
f(z)\delta_z(\zeta),
\end{multline*}
\begin{multline*} - \partial_{{\zeta}}(e^{\overline{k\zeta}-k\zeta }\widetilde{g})=
-\frac{1}{k}\partial_{\zeta}\left(\frac{e^{\overline{k\zeta}-k\zeta} g}{(z-\zeta)(\overline{\zeta} - \overline{w})}   \right) =\\ 
\frac{e^{\overline{k\zeta}-k\zeta}}{(z-\zeta)(\overline{\zeta} - \overline{w})}
\left(1 - \frac{1}{k(z - \zeta )} - \frac{1}{k}\partial_{\zeta}\right) g - 
\frac{1}{k}\frac{e^{\overline{k\zeta}-k\zeta}}{(z-\zeta)}\pi g(w)\delta_w(\zeta).
\end{multline*}
Hence,
\begin{multline}
\partial_{\overline{\zeta}}(e^{\overline{k\zeta}-k\zeta}\widetilde{f}) - \partial_{\zeta}(e^{\overline{k\zeta}-k\zeta}\widetilde{g}) =\\ 
\frac{e^{\overline{k\zeta}-k\zeta}}{(z-\zeta)(\overline{\zeta} - \overline{w})}\left[ \left(1 - \frac{1}{\overline{k}(\overline{\zeta} -\overline{w})} + \frac{1}{\overline{k}}\partial_{\overline{\zeta}}\right)f  + \left(1 - \frac{1}{k(z-\zeta)} - \frac{1}{k}\partial_{\zeta}\right)g \right] \\
- \frac{\pi e^{\overline{k\zeta}-k\zeta}}{\overline{k}(\overline{\zeta} -\overline{w})}f\delta_z(\zeta) - \frac{\pi e^{\overline{k\zeta}-k\zeta}}{k(z-\zeta)}g\delta_{w}(\zeta)
\label{0p4} 
\end{multline}

We would like to have 
\begin{equation*}
\left(1 - \frac{1}{\overline{k}(\overline{\zeta} -\overline{w})} + 
\frac{1}{\overline{k}}\partial_{\overline{\zeta}}\right)f(\zeta)  + \left(1 - \frac{1}{k(z-\zeta )} - \frac{1}{k}\partial_{\zeta}\right)g(\zeta) = 1.
\end{equation*}
We start by constructing a partition 
$ 1 =\chi_{w} + \chi_{z} + r_{w,z}$, where $r_{w,z}$ 
is supported in a region $|\zeta - z|$, $|\zeta - 
w|<\mathcal{O}(\frac{1}{|k|})$ and then solve, up to asymptotic errors,
\begin{equation}
\left(1 - \frac{1}{\overline{k}(\overline{\zeta} -\overline{w})} + \frac{1}{\overline{k}}\partial_{\overline{\zeta}}\right)f = \chi_w,
\label{1p4}
\end{equation}
\begin{equation}
\left(1 - \frac{1}{k(z-\zeta)} - 
\frac{1}{k}\partial_{\zeta}\right)g = \chi_z.
\label{2p4}
\end{equation}
Put
\begin{equation}\label{AB.1}
  d(z,w, k) = |z-w| + \frac{1}{|k|}, \end{equation}
\begin{prop}\label{AB1}
  Let
  \begin{equation}\label{AB.2}
\widehat{d}(z,w,\zeta ,k)=d(z,w,k)+\left| \zeta -\frac{z+w}{2} \right|
\end{equation}
and notice that $\widehat{d}(z,w,\zeta ,k)$ is uniformly of the same
order of magnitude as $d(z,w,k)+|\zeta -z|$ and $d(z,w,k)+|\zeta -w|$,
since $|z-w|\le d(z,w,k)$.

\par For all $(w,z,k)\in \mathbb{C}^3$ with $|k|\ge 1$,
there exist $\chi _w, \chi _z\in C^\infty (\mathbb{C})$, $r_{w,z}\in
C_0^\infty (\mathbb{C})$ such that
\begin{equation}\label{AB.3}\chi _w,\chi _z,r_{w,z}\ge 0,\end{equation}
\begin{equation}\label{AB.4}\chi _w+\chi _z+r_{w,z}=1,\end{equation}
\begin{equation}
  \label{AB.5}
  \begin{split}
|\zeta - z|&\geq\frac{1}{\mathcal{O}(1)}\widehat{d}(z,w,\zeta, k) \text{ on } \mathrm{supp\,}\chi_z,\\
|\zeta - w|&\geq\frac{1}{\mathcal{O}(1)}\widehat{d}(z,w, \zeta, k)\text{ on }\mathrm{supp\,}\chi_w,
\end{split}
\end{equation}
\begin{equation}\label{AB.6}r_{w,z}\hbox{ has its support  in }\left\{\zeta
  \in \mathbb{C};\, |\zeta -z|,\, |\zeta -w|\le \frac{3}{2|k|}\right\}
  ,\end{equation}
\begin{equation}\label{AB.7}
\nabla^\alpha_\zeta\chi_z =
\mathcal{O}\left(\widehat{d}^{-|\alpha |}\right), \quad
\nabla^\alpha_\zeta\chi_{w} = \mathcal{O}\left(
  \widehat{d}^{-|\alpha |}\right),\quad \forall \alpha\in 
  \mathbb{N}^{2}, \quad 
  |\alpha|=|\alpha|_{l^{1}}=\alpha_{1}+\alpha_{2},
\end{equation}
\begin{equation}\label{AB.8}
\nabla^\alpha_\zeta r_{z,w} = 
\mathcal{O}_{N,\alpha}(1)\widehat{d}^{-|\alpha|}(|k|\widehat{d})^{-N}, \forall \alpha\in\mathbb{N}^{2}, \forall N\in \mathbb{N}.
\end{equation}
These estimates are uniform with respect to $w,z,k$.
\end{prop}
\begin{proof} Put
$$\widetilde{\chi}_z(\zeta) = (1-\psi^0)\left(\frac{\zeta - z}{d}\right),\ \
\widetilde{\chi}_w(\zeta ) = (1-\psi^0)\left(\frac{\zeta - w}{d}\right),
$$
where $\psi^0\in C^\infty_0(D(0,1/3))$ is real valued with $1_{D(0,1/4)}\leq\psi^0\leq 1_{D(0,1/3)}$. Clearly $\widetilde{\chi}_z(\zeta) + \widetilde{\chi}_z(\zeta)\geq 0$ and
\begin{equation}\label{AB.9}
\widetilde{\chi}_z(\zeta) + \widetilde{\chi}_z(\zeta) = 0 \Rightarrow 
\begin{cases}
|\frac{\zeta-z}{d}|\leq\frac{1}{3}\\
|\frac{\zeta-w}{d}|\leq\frac{1}{3}
\end{cases}
\Leftrightarrow
\begin{cases}
|\zeta-z|\leq\frac{1}{3}|z-w| +\frac{1}{3}\frac{1}{|k|},\\
|\zeta-w|\leq\frac{1}{3}|z-w| +\frac{1}{3}\frac{1}{|k|}.
\end{cases}
\end{equation}
The last inequalities imply that
$|z-w|\leq\frac{2}{3}|z-w|+\frac{2}{3|k|}$, so
$$|z-w|\leq \frac{2}{|k|},\ \ |\zeta -z|,\, |\zeta -w|\le \frac{1}{|k|}.$$ 

We have $\nabla^{\alpha}_\zeta\widetilde{\chi}_z$, $\nabla^{\alpha}_\zeta\widetilde{\chi}_w = \mathcal{O}(d^{-|\alpha|})$ and 
\begin{equation}
  \label{0p5}
  \begin{split}
|\zeta - z|&\geq\frac{1}{\mathcal{O}(1)}d(z,w,k) \text{ on } \mathrm{supp\,}\widetilde{\chi}_z,\\
|\zeta - w|&\geq\frac{1}{\mathcal{O}(1)}d(z,w,k)\text{ on }\mathrm{supp\,}\widetilde{\chi}_w,
\end{split}
\end{equation}
where we can replace $d(z,w,k)$ with $\widehat{d } (z,w,\zeta
  ,k)$.
Moreover, since $\widetilde{\chi}_\cdot(\zeta) =1$ when $|\cdot -\zeta|\geq 
\frac{1}{3}d$, we have
$$
\nabla^\alpha_\zeta\widetilde{\chi}_z =
\mathcal{O}\left((d+|z-\zeta|)^{-|\alpha |}\right), \quad
\nabla^\alpha_\zeta\widetilde{\chi}_{w} = \mathcal{O}\left(
  (d+|w-\zeta|)^{-|\alpha |}\right).
$$
Hence as noted in the statement of the proposition,
\begin{equation}\label{1p5}\nabla _{\zeta }^\alpha \widetilde{\chi
  }_z,\ \nabla _{\zeta }^\alpha \widetilde{\chi
  }_w =\mathcal{ O}(\widehat{d}^{-|\alpha |}). \end{equation}

Let $\Psi^1\in C^\infty_0(D(0, \frac{3}{2}))$ satisfy
$1_{D(0,1)}\leq\Psi^1\leq 1$ and put $\widetilde{r}_{z,w} =
\Psi^1(|k|(\zeta - z))\Psi^1(|k|(\zeta - w))$. Then, by (\ref{AB.9}) and
the subsequent observation,
\begin{equation}
f:=\widetilde{\chi}_z(\zeta)+\widetilde{\chi}_w(\zeta)+\widetilde{r}_{z,w}\asymp 1,
\label{2p5}
\end{equation}
and by construction the last term has its support in $$\left\{ 
\zeta\in\mathbb{C}; |\zeta-z|,~ |\zeta
-w|\leq\frac{3}{2|k|}\right\},$$ so that $|z-w|\leq
\frac{\mathcal{O}(1)}{|k|}$, $\widehat{d}(z,w,\zeta ,k) \asymp
\frac{1}{|k|}$ on $\mathrm{supp}~\widetilde{r}_{z,w} $. It follows
that $\widetilde{r}_{z,w} = \mathcal{O}_{N}(1)(|k|\widehat{d})^{-N},~
\forall N\ge 0$ and more generally
\begin{equation}
\nabla^\alpha_\zeta\widetilde{r}_{z,w} = \mathcal{O}_{N,\alpha}(1)\widehat{d}^{-|\alpha|}(|k|\widehat{d})^{-N}
\label{3p5}.
\end{equation}
From   (\ref{1p5}) we get
\begin{equation}
\nabla^\alpha_\zeta f,~\nabla^\alpha_\zeta\frac{1}{f} = \mathcal{O}(1)\widehat{d}^{-|\alpha|}.
\label{5p5}
\end{equation}
Put 
\begin{equation}
\chi_z = \widetilde{\chi}_z/f,\quad  \chi_w = \widetilde{\chi}_w/f,\quad 
r_{z,w} =\widetilde{r}_{z,w}/f .
\label{6p5}
\end{equation}
Then we have
\begin{equation}
\chi_z + \chi_w + r_{z,w} = 1,
\label{7p5}
\end{equation}
 and (\ref{AB.3})--(\ref{AB.8}) follow. \end{proof}

We now return to the problem (\ref{1p4}) -- (\ref{2p4}). As a first 
approximate solution, we take $f^0 = \chi_w$, $g^0 = \chi_z$ and 
treat the other terms in the LHSs as perturbations. We then get $f,g$ 
as formal Neumann series sums
\begin{equation}
f = \sum^\infty_0\left(\frac{1}{\overline{k}(\overline{\zeta}
    -\overline{w})}  -
  \frac{1}{\overline{k}}\partial_{\overline{\zeta}}\right)^\nu \chi_w, \quad
g = \sum^\infty_0\left(\frac{1}{k(z -\zeta)}  +
  \frac{1}{k}\partial_{\zeta}\right)^\nu \chi_z .
\label{23p6}
\end{equation}        
Recall (\ref{AB.5}) and (\ref{AB.6}).
Then 
\begin{equation*}
\begin{split}
\nabla^{\nu _0}_\zeta\frac{1}{(\overline{\zeta} - \overline{w})^{\nu _1}}\partial_{\overline{\zeta}}^{\nu _2}\chi_w &= \mathcal{O}(\widehat{d}^{-(\nu _0+\nu _1+\nu _2)}),\\ 
~\\
\nabla^{\nu _0}_\zeta\frac{1}{(z-\zeta )^{\nu _1}}\partial_{{\zeta}}^{\nu _2}\chi_z &= \mathcal{O}(\widehat{d}^{-(\nu _0+\nu _1+\nu _2)}). 
\end{split}
\end{equation*}
Thus
\begin{equation}
\left.
\begin{array}{ll}
  &\displaystyle\nabla_{\zeta}^{\nu _0}\left(
    \frac{1}{\overline{k}(\overline{\zeta} -
    \overline{w})} - \frac{1}{\overline{k}}\partial_{\overline{\zeta}}\right)^\nu \chi_w\\
  &\displaystyle\nabla_{\zeta}^{\nu _0}\left( \frac{1}{k(z-\zeta )} + \frac{1}{k}\partial_{\zeta}\right)^\nu \chi_z
\end{array}\right \} = \mathcal{O}(\widehat{d}^{-\nu _0}(k\widehat{d})^{-\nu }).
\label{5p6}
\end{equation}
For $N\in \mathbb{N}$ define $f_N,~g_N$ as in (\ref{23p6}) but with finite sums
\begin{equation}
f_N = \sum^N_0\left(\frac{1}{\overline{k}(\overline{\zeta}
    -\overline{w})}  -
  \frac{1}{\overline{k}}\partial_{\overline{\zeta}}\right)^\nu \chi_w, \quad
g_N = \sum^N_0\left(\frac{1}{k(z -\zeta)}  +
  \frac{1}{k}\partial_{\zeta}\right)^\nu \chi_z .
\label{23p7}
\end{equation}
Notice that by (\ref{5p6}) \begin{equation}\label{3.5p7}f_N,\, g_N = \mathcal{O}(1) .\end{equation} Then c.f. (\ref{1p4}), (\ref{2p4})
\begin{equation}
\left(1 - \frac{1}{\overline{k}(\overline{\zeta} - \overline{w})} + \frac{1}{\overline{k}}\partial_{\overline{\zeta}}\right)f_N = \chi_w - \left(  \frac{1}{\overline{k}(\overline{\zeta} - \overline{w})} - \frac{1}{\overline{k}}\partial_{\overline{\zeta}}\right)^{N+1} \chi_w =:\chi_w - S^{N+1}_w
\label{4p7}
\end{equation}
\begin{equation}
\left(1 - \frac{1}{k(z-\zeta )} -
  \frac{1}{k}\partial_{\zeta}\right)g_N = \chi_z - \left(
  \frac{1}{k(z-\zeta ) } + \frac{1}{k}\partial_{\zeta}\right)^{N+1} \chi_z =:\chi_z - T^{N+1}_z
\label{5p7}
\end{equation}
where
\begin{equation}
\nabla^{\nu _0}_\zeta S^{N+1}_w,~\nabla^{\nu _0}_\zeta T^{N+1}_z = \mathcal{O}(\widehat{d}^{-\nu _0}(k\widehat{d})^{-N-1}).
\label{6p7}
\end{equation}
Now, combine (\ref{4p7}) - (\ref{6p7}) with  (\ref{7p5}) and (\ref{3p5}) (valid also for $r_{z,w}$) to get 
\begin{multline}
\left(1 - \frac{1}{\overline{k}(\overline{\zeta } - \overline{w})} + \frac{1}{\overline{k}}\partial_{\overline{\zeta}}\right)f_N + \left(1 - \frac{1}{k(z - \zeta)} - \frac{1}{k}\partial_{\zeta}\right)g_N =\\ 
 1-S_w^{N+1} - T_z^{N+1} - r_{z,w} =: 1- r^{N+1}
\label{7p7}
\end{multline}
where
\begin{equation}
\nabla^{\nu _0}_\zeta r^{N+1} = \mathcal{O}(\widehat{d}^{-\nu _0}(k\widehat{d})^{-N-1}), \quad \nu _0\in \mathbb{N}
\label{8p7}.
\end{equation}
We use this in the discussion after (\ref{1p3}). With 
\begin{equation*}
\widetilde{f}_N  = \frac{f_N}{\overline{k}(z-\zeta)(\overline{\zeta} - \overline{w})}, \quad 
\widetilde{g}_N  = \frac{g_N}{k(z-\zeta)(\overline{\zeta} - \overline{w})} 
\end{equation*}
we get (cf. (\ref{0p4})) 
\begin{multline*}
d\left(e^{\overline{k\zeta} -k\zeta}(\widetilde{f}_N d\zeta + \widetilde{g}_N \overline{d\zeta}  )\right) =\\
\left[\frac{e^{\overline{k\zeta} -k\zeta}}{(z-\zeta)(\overline{\zeta} - \overline{w})} - \frac{e^{\overline{k\zeta} -k\zeta}r^{N+1}}{(z-\zeta)(\overline{\zeta} - \overline{w})}  -
\frac{\pi e^{\overline{k\zeta} -k\zeta}f_N\delta_z(\zeta)}{\overline{k}(\overline{\zeta} - \overline{w})} -
\frac{\pi e^{\overline{k\zeta} -k\zeta}g_N\delta_w(\zeta)}{k(z- \zeta)}  
 \right]\overline{d\zeta}\wedge d\zeta.
\end{multline*}
We use this in (\ref{1p3}) and apply Stokes' formula:
\begin{equation}\label{1p8}\begin{split}
    K(z,w) =& \frac{\sigma
      e^{kw - \overline{kw}}}{(2\pi)^22\mathrm{i}}\left[\underbrace{\iint_\Omega
      \frac{e^{\overline{k\zeta} -k\zeta}}{(z-\zeta)(\overline{\zeta}
        - \overline{w})}r^{N+1}\overline{d\zeta}\wedge
      d\zeta}_{=:K_1(z,w)}\right. \\
    &+\underbrace{\frac{2i\pi e^{\overline{kz} -kz}}{\overline{k}(\overline{z} - \overline{w})}f_N(z,w,z)1_{\Omega}(z)}_{=:K_2(z,w)}+
    \underbrace{\frac{2i\pi e^{\overline{kw} -kw}}{k(z -
        w)}g_N(z,w,w) 1_{\Omega}(w)}_{=:K_3(z,w)} \\
    & +\left.\underbrace{\int_{\partial\Omega}\frac{ e^{\overline{k\zeta} -k\zeta}f_N}{\overline{k}(\overline{\zeta} - \overline{w})(z - \zeta)}d\zeta}_{=:K_4(z,w)} +
     \underbrace{\int_{\partial\Omega}\frac{
        e^{\overline{k\zeta} -k\zeta}g_N}{k(\overline{\zeta} -
        \overline{w})(z - \zeta
        )}d\overline{\zeta}}_{=:K_5(z,w)}\right].
\end{split}
\end{equation}
The factor in front of the big bracket is bounded, so it will suffice
to estimate the norms between various weighted $L^p$ spaces of the operators
\begin{equation}\label{defAj}
A_ju(z)  = \int_\Omega K_j(z,w)u(w)L(dw)
\end{equation} 
with integral kernel $K_j$. Recalling (\ref{8p7}), we get
\begin{equation*}
K_1(z,w)  = \frac{\mathcal{O}(1)}{|k|^{N+1}}\int_{\Omega}
\frac{1}{|z-\zeta||w-\zeta|}\frac{1}{\left(|z-w|+
    |\zeta-\frac{z+w}{2}| +\frac{1}{|k|}\right)^{N+1}}L(d\zeta ).
\end{equation*}
Consider separately the integrals over
$\Omega_1 = \{\zeta\in\Omega;~|\zeta-w|\geq|\zeta - z|\}$ and
$\Omega_2 = \{\zeta\in\Omega;~|\zeta-w|<|\zeta - z|\}$. The two
integrals can be handled similarly, and we only need to consider the
first case $|\zeta - w|\geq|\zeta-z|$. Here
$|\zeta-w|\geq\frac{1}{2}|z-w|$ and the corresponding integral is
\begin{multline*}
  \leq \frac{\mathcal{O}(1)}{|z-w|}\frac{1}{|k|^{N+1}}
  \int_{\Omega_1} \frac{1}{|z-\zeta|}
  \frac{1}{\left(|z-w|+ |\zeta-z| +\frac{1}{|k|}\right)^{N+1}}L(d\zeta
  ) \\
  \leq \frac{\mathcal{O}(1)}{|z-w|}\frac{1}{|k|^{N+1}}
  \int_{\mathbb{C}} \frac{1}{|\zeta|}\frac{1}{\left(|\zeta|
      +\lambda\right)^{N+1}}L(d\zeta ), \end{multline*}
where $\lambda:= |z-w|+1/|k|$. Putting
$$\zeta =
  \lambda\widetilde{\zeta},\ \ L(d\zeta) =
  \lambda^2L(d\widetilde{\zeta}),$$
  gives the upper bound
  \begin{multline*} 
  = \frac{\mathcal{O}(1)}{|z-w|}\frac{1}{|k|^{N+1}}\frac{\lambda^2}{\lambda^{N+2}}\int_{\mathbb{C}} \frac{1}{|\widetilde{\zeta}|}\frac{1}{\left(|\widetilde{\zeta}| +1\right)^{N+1}}L(d\zeta)\\
  = \mathcal{O}(1)\frac{1}{|k||z-w|\left( |k|\lambda \right)^N} = \frac{\mathcal{O}(1)}{|k||z-w|}\frac{1}{\left(|k||z-w|+1\right)^N}\\
  = \mathcal{O}(1)f_{k,N}(z-w).
\end{multline*}
Here
\begin{equation*}
f_{k,N}(z) = \frac{1}{|k||z|(1+|k||z|)^N}, 
\end{equation*} 
and
\begin{multline*}
\int f_{k,N}(z)L(dz) = \int\frac{1}{|k||z|(1+|k||z|)^N}L(dz) =\\ 
\frac{1}{|k|^2}\int\frac{1}{|\widetilde{z}|(1+|\widetilde{z}|)^N}L(d\widetilde{z}) = \mathcal{O}(1)\frac{1}{|k|^2}.
\end{multline*}
We deduce that $K_1$ is bounded  by an $L^1$ convolution kernel, hence
\begin{equation*}
A_1 = \mathcal{O}(1)\frac{1}{|k|^2}: L^p\to L^p, \quad 1\leq p\leq \infty.
\end{equation*}
By (\ref{3.5p7}) we have
$K_2(z,w),~ K_3(z,w) = \frac{\mathcal{O}(1)}{|k||z-w|}$ and it follows
(here we integrate over $\Omega$ in (\ref{defAj})) that for every
bounded set $V\subset \mathbb{C}$,
\begin{equation}
1_VA_2,\ 1_VA_3= \mathcal{O}(|k|^{-1}): L^p\to L^p, \quad 1\leq p\leq \infty.
\label{1p9}
\end{equation}
In fact, $1/|z|$ is integrable on every bounded set.

\par
We next estimate the contribution to $A_2,~ A_3$ from $|z|\gg 1$. For
$j=2,3$ let $C_j=1_{\mathbb{C}\setminus \mathrm{neigh\,}(\overline{\Omega
  })}A_j$. Then 
\begin{equation*}
C_ju(z) = \int_\Omega \widetilde{K}_j(z,w)u(w)L(dw),\ \  |\widetilde{K}_j(z,w)|\leq \frac{1}{\langle z\rangle|k|},
\end{equation*}
\begin{equation*}
|C_ju(z)|\leq \mathcal{O}(1)\langle
z\rangle^{-1}|k|^{-1}||u||_{L^1(\Omega )}\leq\mathcal{O}(1)\langle
z\rangle^{-1}|k|^{-1}||u||_{L^p(\Omega )},\ 1\leq p\leq\infty .
\end{equation*}
Here
$$
\int \langle z\rangle^{-q}L(dz)<\infty \text{ if } q>2,$$
so $$C_j = \mathcal{O}\left(1/|k|\right):L^p\to L^q \text{ if } 1\leq p\leq\infty,~q>2.
$$  
If $q\leq 2$, $\epsilon >0$. Then,
\begin{equation*}
\|\langle \cdot \rangle^{-1-\epsilon }\|_{L^q}^q=\int \langle z \rangle^{-(1+\epsilon)q}L(dz)<\infty
\end{equation*}
iff $(1+\epsilon)q>2$, i.e.\ iff $\epsilon>2/q -1$. When $q=2$, this
amounts to $\epsilon>0$. We conclude that 
\begin{equation*}
C_j = \mathcal{O}_\epsilon(1/|k|):L^p\to \langle\cdot\rangle^{\epsilon}L^q
\text{ when } 1\leq p\leq \infty,~1\leq q\leq 2,~ \epsilon>\frac{2}{q} - 1.
\end{equation*}
Hence for $j=2,3$:
$$
A_j=\mathcal{O}(1/|k|):\begin{cases}L^q\to L^q,\hbox{ when }q>2,\\
L^q\to \langle \cdot \rangle^\epsilon L^q, \hbox{ when }1\le q\le 2,\
\epsilon >2/q-1.
\end{cases}
$$

We next estimate $A_4,~A_5$ with kernels $K_4,~K_5$ in (\ref{1p8}). It suffices to treat $A_4, K_4$ since the expression for $K_5$ is very similar. Let
\begin{equation*}
\Gamma_z = \{\zeta\in \partial\Omega;|\zeta-z|\geq \frac{1}{2}|w-z|\}, 
\quad \Gamma_w = \{\zeta\in \partial\Omega;|\zeta-w|\geq \frac{1}{2}|z-w|\},
\end{equation*}
so that $\partial \Omega\subset\Gamma_z\cup\Gamma_w$. Then 
\begin{equation}\label{1p10}\begin{split}
|K_4(z,w)|&\leq \frac{\mathcal{O}(1)}{|k|}\left( \int_{\Gamma_z} \frac{|d\zeta|}{|z-\zeta||\zeta-w|} + \int_{\Gamma_w} \frac{|d\zeta|}{|z-\zeta||\zeta-w|}\right)\\
&\leq \frac{\mathcal{O}(1)}{|k|}\left( \frac{1}{|w-z|}\int_{\partial \Omega} \frac{|d\zeta|}{|\zeta-w|} + \frac{1}{|w-z|}\int_{\partial \Omega} \frac{|d\zeta|}{|z-\zeta|}\right)\\
&\leq \frac{\mathcal{O}(1)}{|k|}\left( \frac{1}{|w-z|}G(w) + \frac{1}{|w-z|}G(z)\right),
\end{split}
\end{equation}
where
\begin{equation}\label{2p10}
G(z)  = 
\begin{cases}
1 + |\ln d(\partial\Omega, z)|, & z\in \mathrm{neigh}(\overline{\Omega}, \mathbb{C}), \\
1/\langle z\rangle,   & z\in\mathbb{C}\setminus \mathrm{neigh}(\overline{\Omega}, \mathbb{C}) ,
\end{cases}
\end{equation} 
 and $d(\partial\Omega, w)$ denotes the distance between $\partial\Omega$ and $w$. 

$|\cdot -z|^{-\alpha }$ and $G^\beta $ are integrable on
any bounded set when $\beta \ge 0$, $0<\alpha <2$. Choose $\alpha =3/2$ 
and write $\frac{|G(w)|}{|w-z|}$ as the geometric mean
$\left(\frac{1}{|w-z|}^{\frac{3}{2}}\right)^{\frac{2}{3}}\left(G(w)^3
\right)^{\frac{1}{3}} $. Using that geometric means are bounded by
the arithmetic ones, we get 
\begin{equation*}
\frac{1}{|w-z|}G(w)\leq\frac{2}{3}\frac{1}{|w-z|^{3/2}} + 
\frac{1}{3}G(w)^3.
\end{equation*}
Using this and the corresponding estimate with $G(z)$ in (\ref{1p10}) we get
\begin{equation}
1_{\mathrm{neigh}(\overline{\Omega})}(z)|K_j(z,w)|1_\Omega (w)\leq\frac{\mathcal{O}(1)}{|k|}\frac{1}{|z-w|^{\frac{3}{2}}}+\frac{\mathcal{O}(1)}{|k|}G(w)^3 + \frac{\mathcal{O}(1)}{|k|}G(z)^3, 
\label{3p10} 
\end{equation}
when $j = 4$. Clearly the same estimate holds when $j=5$. The first 
term is an $L^1$-convolution kernel (neglecting a region $|z| \gg 1$ 
and recalling that $w\in \Omega $), gives rise to an operator
\begin{equation*}
\mathcal{O}(1/|k|):L^p(\Omega )\to
  L^p(\mathrm{neigh\,(\overline{\Omega} ) ),} \quad p\in[1,\infty]
\end{equation*}
for any bounded neighborhood of $\overline{\Omega }$.
By the H\"older inequality and the fact that $G^\alpha\in L^1, 
\forall \alpha>0$, we see that the second term gives rise to an operator
\begin{equation*}
\mathcal{O}(1/|k|): L^p\to L^\infty,\ \forall p>1.
\end{equation*}
Similarly the third term gives rise to an operator 
\begin{equation*}
\mathcal{O}(1/|k|): L^1(\Omega )\to
  L^q(\mathrm{neigh\,}(\overline{\Omega} ) ), \quad q\in[1,\infty[
\end{equation*}
for any bounded neighborhood of $\overline{\Omega }$. 
Recalling again that we work on a bounded subset  of $\mathbb{C}$ where $L^p\subset L^{q}$ for $1\leq q\leq p \leq \infty$ we conclude that
\begin{equation}
1_{\mathrm{neigh}(\overline{\Omega})}A_4,~1_{\mathrm{neigh}(\overline{\Omega})}A_5 = \mathcal{O}(1/|k|): L^p\to L^p, \quad 1<p<\infty.
\label{1p11}
\end{equation}
For $z\in\mathbb{C}\setminus \mathrm{neigh\,}(\overline{\Omega})$, $w\in \Omega$, (\ref{1p10}) and (\ref{2p10}) give
\begin{equation*}
|K_4(z, w)|\leq\frac{\mathcal{O}(1)}{|k|\langle z\rangle}(1+ |\ln d(w, \partial\Omega) |).
\end{equation*}
This is the same estimate as for $K_2, K_3$ except that the
$1_\Omega (w)$, belonging to all $L^{p'}$ with
$1\leq p' \leq +\infty$, is replaced by
$(1+|\ln(w, \partial\Omega)|)1_{\Omega}(w)$ belonging to all $L^{p'}$
with $1\leq p'<+\infty$. The estimates for $1_{\mathbb{C}\setminus
  \mathrm{neigh\,}(\overline{\Omega })}A_j$, $j=2,3$ extend to
$1_{\mathbb{C}\setminus \mathrm{neigh\,}(\overline{\Omega})}A_4$,
$1_{\mathbb{C}\setminus \mathrm{neigh\,}(\overline{\Omega})}A_5 : L^p\to L^q$,
for $1< p \leq \infty$: For $j = 4, 5,$
\begin{equation*}
1_{\mathbb{C}\setminus \mathrm{neigh\,}(\overline{\Omega})}A_j =
\mathcal{O}(1/|k|):
\begin{cases}L^p\to L^q,\quad 1<p\leq \infty,\ q>2,\\ L^p\to \langle \cdot\rangle^{\epsilon}L^q,\ 1<p\leq \infty,\  1\leq q\leq 2,\ \epsilon>\frac{2}{q}  - 1.\end{cases}
\end{equation*}
Hence with (\ref{1p11})
\begin{equation*}
A_j = \mathcal{O}(1/|k|):
\begin{cases}
L^q\to L^q, &  2<q<+\infty,\\
L^q\to \langle \cdot\rangle^\epsilon L^q, &  1<q \leq 2, \epsilon>\frac{2}{q} -1.
\end{cases}
\end{equation*}
Combining the estimates for $A_j$, $1\leq j \leq 5$, we get 
\begin{theo}\label{AB2}
\begin{equation}
AB = \mathcal{O}(1/|k|):\begin{cases}
L^q\to L^q, &  2<q<+\infty,\\
L^q\to \langle \cdot\rangle^\epsilon  L^q, &  1<q \leq 2,\ \epsilon>\frac{2}{q} -1.
\end{cases}
\label{2p11}
\end{equation}
\end{theo}

\section{Back to the reflection coefficient}\label{bref}
\setcounter{equation}{0}
Recall (\ref{reflc})
\begin{equation}\label{bref.1}
    \overline{R} = 
    \frac{2\sigma}{\pi}\int_{\mathbb{C}}^{}\mathrm{e}^{kz-\overline{kz}}\overline{q}(z)\phi_{1}(z;k)L(\mathrm{d}z),
  \end{equation}
  where $\phi _1=1+\phi ^1_1$ and we assume $q=1_\Omega $ where
  $\Omega \Subset \mathbb{C}$ is open with strictly convex smooth boundary. 
  $\phi _1^1$ is given by (\ref{sy.41}),
  \begin{equation}\label{bref.2}
\phi _1^1= (1-AB)^{-1}AB(1)=AB(1)+(AB)^2(1)+...
\end{equation}
and $A$, $B$ are given in (\ref{sy.16}), now with $q=1_\Omega $
\begin{equation}\label{bref.2.5}
\begin{cases}A=E\widehat{\tau }_{-\omega }\frac{h}{2}1_\Omega ,\\
B=\sigma F\widehat{\tau }_\omega \frac{h }{2}1_\Omega\ (=\sigma\overline{A})
\end{cases} .
\end{equation}

\par By Theorem \ref{AB2} we know that $1_\Omega AB=\mathcal{
 O}(1/|k|):L^2(\Omega )\to L^2(\Omega )$ and we shall frequently use
that $A=A\circ 1_\Omega $, $B=B\circ 1_\Omega $. Combining
(\ref{bref.1}), (\ref{bref.2}), we get (assuming $\sigma =1$ for
simplicity)
\begin{equation}\label{bref.3}
\overline{R}=\frac{2}{\pi }\sum_{\nu =0}^{N-1}\int_\Omega
e^{kz-\overline{kz}}(AB)^\nu (1)L(dz)+\mathcal{ O}(|k|^{-N}),
\end{equation}
for every $N=1,2,...$. Here
$$\left|\int_\Omega e^{kz-\overline{kz}}(AB)^\nu (1)L(dz)\right|\le
\mathrm{vol\,}(\Omega )^{1/2}\| (AB)^\nu (1)\|_{L^2(\Omega )}\le C
(C/|k|)^\nu $$
and we shall see that this estimate can be improved by using more
information about $A=A_k$ from Section 5 in \cite{KlSjSt20} in the 
case when $\partial \Omega$ is analytic. In remark \ref{remark} we 
explain how to extend the discussion to the case when $\partial\Omega$ 
is merely smooth. 

\par First, recall from (\ref{bref.2.5}) and the explicit formula for
the fundamental solution of $\overline{\partial } $ appearing in $E$,
that
\begin{equation}\label{bref.4}
Au(z)=\frac{1}{2\pi }\int_\Omega e^{-kw+\overline{kw}}\frac{1}{z-w}u(w)L(dw),
\end{equation}
or
\begin{equation}\label{bref.5}
A=A_k=A_0\circ
e^{-k\cdot +\overline{k\cdot }},\end{equation}
where
\begin{equation}\label{bref.6}
A_0u(z)=\frac{1}{2\pi }\int_\Omega \frac{1}{z-w}u(w)L(dw).
\end{equation}
$A_0$ is anti-symmetric for the standard bilinear scalar product on
$L^2$; $A_0^{\mathrm{t}}=-A_0$, so the transpose of $A_k$ is given by
\begin{equation}\label{bref.7}
A_k^{\mathrm{t}}=-e^{-k\cdot +\overline{k\cdot }}A_0=-e^{-k\cdot
  +\overline{k\cdot }}A_{-k}e^{-k\cdot +\overline{k\cdot }} ,
\end{equation}
and $B_k$ is the complex conjugate of $A_k$ (here $\sigma=1$ for 
simplicity):
$$
B_k=\overline{A_k}.
$$

\par In \cite{KlSjSt20} we studied the function
\begin{equation}\label{bref.8}
f(z,k)=2\pi A_k(1_\Omega )(z)=\int_\Omega 
e^{-kw+\overline{kw}}\frac{1}{z-w}L(dw),
\end{equation}
when $\partial\Omega$ is analytic.  

\par By (5.32), (5.31) in \cite{KlSjSt20}, we have
\begin{equation}\label{bref.9}
  f(z,k )=\frac{1}{2i\overline{k}}F(z)+(\pi
/\overline{k})\left(e^{-iu(z,k )}(1_{\Omega _-}(z)-1_{\Omega
    _+}(z))+e^{-kz+\overline{kz}}1_\Omega (z) \right),
\end{equation}
where
\begin{equation}\label{bref.10}
F(z)=F_\Gamma (z)=\int_\Gamma \frac{1}{z-w}e^{-iu(w,k )}dw.
\end{equation}
Here the deformation $\Gamma $ of $\partial \Omega $ and the domains
$\Omega _\pm$ are defined in \cite[Section 5]{KlSjSt20}, see 
Fig.~\ref{contour}. Further $-iu(\cdot
,k)$ is the holomorphic extension of $-k\cdot +\overline{k\cdot }$
from $\partial \Omega $ to a neighborhood.
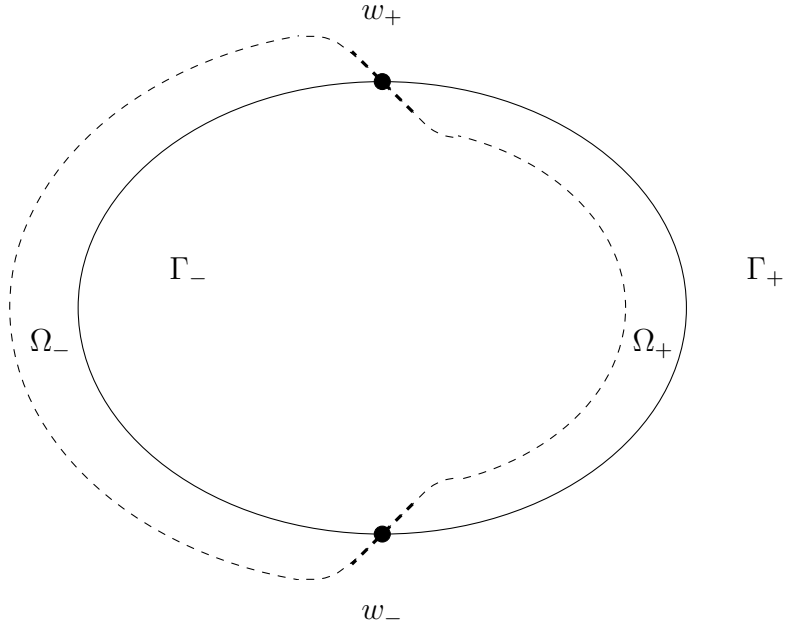
\begin{figure}[htb!]
\begin{tikzpicture}



\draw (0,0)  circle [x radius=4 cm, y radius = 3cm];

\draw[very thick, black, dashed] (0.4,2.6)--(-.4,3.4);
\draw[very thick, black, dashed] (0.4,-2.6)--(-.4,-3.4);

\node (node) {};
\node [above left = 0 cm and 2 cm of node](label){$\Gamma_-$};
\node [above right = 0 cm and 4.5 cm of node](label){$\Gamma_+$};
\node [below right = 0 cm and 3 cm of node](label){$\Omega_+$};
\node [below left = 0 cm and 3.8 cm of node](label){$\Omega_-$};

\begin{scope}
    \clip (-1.1, -3.6) rectangle (-5, 8);
    \draw(0,0)[dashed] circle   [x radius=4.9 cm, y radius = 3.7cm];
\end{scope}

\begin{scope}
    \clip (1.,-3) rectangle (6,6);
    \draw(-.0,0)[dashed] circle   [x radius=3.2 cm, y radius = 2.4cm];
\end{scope}

\filldraw 
(0,+ 3) circle (3pt) node[align=center, above] {$w_+$\\ \vspace{.2cm}}

(0, -3) circle (3pt) node[align=center, below] {\\ \\ $w_-$} ;
\draw[dashed,rounded corners=5] (-1,-3.6)--(-.7,-3.6)--(-.3, -3.3)--(0,-3)--(.3,-2.7)--(.7,-2.3) --(1,-2.25);
\draw[dashed,rounded corners=5] (-1,3.6)--(-.7,3.6)--(-.3, 3.3)--(0,3)--(.3, 2.7)--(.7,2.3) --(1,2.25);
\end{tikzpicture}

 \caption{Real analytic strictly convex boundary $\partial \Omega$ (solid) of some domain 
 $\Omega$  and the deformed contour $\Gamma$ (dashed) for this 
 example.  }
 \label{contour}
\end{figure}

\par In (5.69), (5.70) in \cite{KlSjSt20} we have seen that
\begin{equation}\label{bref.11}
F(z)=\mathcal{ O}(1)(|z-w_+(k)||k|^{1/2}+1)^{-1}+\mathcal{ O}(1)(|z-w_-(k)||k|^{1/2}+1)^{-1},
\end{equation}
where $w_+(k),\, w_-(k)\in \partial \Omega $ are the North and South
poles, determined by the fact that the interior unit normal of
$\partial \Omega $ at $w_\pm (k)$ is of the form $c_\pm \omega $ for
some $c_\pm$ with $\pm c_\pm >0$. Here $\omega \in \mathbb{R}^2\simeq
\mathbb{C}$ is determined by $kz-\overline{kz}=\mathrm{i}|k|\Re (z\overline{\omega
})$ for all $z\in z$.

\par We deduced that
\begin{equation}\label{bref.12}
\| F\|_{L^2(\Omega )}=\mathcal{ O}(1)\frac{(\ln |k|)^{1/2}}{|k|^{1/2}},
\end{equation}
\begin{equation}\label{bref.13}
\| F/(2i\overline{k})\|_{L^2(\Omega )}=\mathcal{ O}(1)\frac{(\ln |k|)^{1/2}}{|k|^{3/2}},
\end{equation}
and using (\ref{bref.11}) we now add the observation that 
\begin{equation}\label{bref.14}
\| F/(2i\overline{k})\|_{L^1(\Omega )}=\mathcal{ O}(1)\frac{1}{|k|^{3/2}}.
\end{equation}

We next estimate $(\pi /\overline{k})e^{-iu(\cdot ,k)}1_{\Omega _+}$,
appearing in (\ref{bref.9}). (Notice that $e^{-iu(\cdot ,k)}1_{\Omega
  _-}$ is absent, since we restrict the attention to $\Omega $ and
$\Omega _-\cap \Omega =\emptyset $.) In (5.75) in \cite{KlSjSt20} we
found that
\begin{equation}\label{bref.15}
  \left\|
(\pi
/\overline{k})e^{-iu(\cdot ,k )}1_{\Omega _+}
  \right\|_{ L^2(\Omega )}=\frac{\mathcal{ O}(1)(\ln |k|)^{1/2}}{|k|^{3/2}},
\end{equation}
and we shall now apply the same procedure to the $L^1$-norm. We
restrict the attention to the contribution to the $L^1$-norm from a
neighborhood of one of the poles, say $w_+(k)$. (Away from such
neighborhoods, the estimates are simpler and lead to a stronger
conclusion.) In suitable coordinates $\mu =t+is$ we have
$$(\pi
/\overline{k})e^{-iu(\cdot ,k )}1_{\Omega _+}=\mathcal{
  O}(k^{-1})e^{-|k|ts/C},\ 0\le t\le 1/\mathcal{ O}(1),\ 0\le s\le t.$$
(Away from a neighborhood of $\{w_+(k),w_-(k) \}$ we have the same
estimate, now for $ 1/\mathcal{ O}(1) \le t\le \mathcal{ O}(1)$, $\ 0\le
s\le 1/\mathcal{ O}(1)$.) The contribution to the $L^1$-norm of $(\pi
/\overline{k})e^{-iu(\cdot ,k )}1_{\Omega _+}$ is
\begin{multline*}
\mathcal{ O}(|k|^{-1})\int_0^1\int_0^t e^{-ts|k|/C}dsdt=
\mathcal{ O}(|k|^{-1})\int_0^1\frac{1}{t|k|}\left(1 - e^{-t^2|k|/C}
\right)dt \\ 
=\frac{\mathcal{ O}(1) }{|k|}\left( \int_0^{|k|^{-1/2}}tdt 
  +\int_{|k|^{-1/2}}^1\frac{1}{t|k|} dt \right)=
\frac{\mathcal{ O}(1)\ln |k|}{|k|^2}.
\end{multline*}
The contribution from $\overline{\Omega }\setminus \mathrm{neigh\,}(\{w_+,w_-
\})$ to the $L^1$-norm is $\mathcal{ O}(1)|k|^{-2}$. Hence
\begin{equation}\label{bref.16}
  \left\|
(\pi
/\overline{k})e^{-iu(\cdot ,k )}1_{\Omega _+}
  \right\|_{ L^1(\Omega )}=\frac{\mathcal{ O}(1)\ln |k|}{|k|^2}.
\end{equation}

\par Combining (\ref{bref.9}), (\ref{bref.13}), (\ref{bref.15}), we
get as in \cite{KlSjSt20} that
\begin{equation}\label{bref.17}
A_k(1_\Omega )=\frac{1}{2\overline{k}}e^{-kz+\overline{kz}}1_\Omega +r(z,k),
\end{equation}
i.e.,
\begin{equation}\label{bref.25}
r=\frac{F}{4\pi i\overline{k}}-\frac{1}{2\overline{k}}e^{-iu}1_{\Omega
_+}\hbox{ in }\Omega ,
\end{equation}
where
\begin{equation}\label{bref.18}
r(\cdot ,k)=\mathcal{ O}(1)|k|^{-3/2}(\ln |k|)^{1/2}\hbox{ in }L^2(\Omega
) .
\end{equation}
Using (\ref{bref.14}), (\ref{bref.16}) instead of (\ref{bref.13}),
(\ref{bref.15}), we get
\begin{equation}\label{bref.19}
\| r(\cdot ,k)\|_{L^1(\Omega )}=\mathcal{ O}(1)|k|^{-3/2}.
\end{equation}

\par We now return to the expansion (\ref{bref.3}) for $\overline{R}$,
and start with the term for $\nu =1$. Let
$$
\langle u|v\rangle =\int_\Omega u(z)v(z)L(dz)
$$
denote the bilinear $L^2$ scalar product. We get with $A=A_k$, $B=B_k$
if nothing else is indicated,
\begin{multline*}
\frac{2}{\pi }\int_\Omega e^{kz-\overline{kz}}AB(1_\Omega
)(z)L(dz)=\frac{2}{\pi} \langle AB(1_\Omega )|e^{k\cdot
  -\overline{k\cdot }}1_\Omega \rangle \\
=\frac{2}{\pi }\langle B(1_\Omega )|A^{\mathrm{t}}e^{k\cdot
  -\overline{k\cdot }}(1_\Omega )\rangle = -\frac{2}{\pi }\langle
B_k(1_\Omega )| e^{-k\cdot +\overline{k\cdot }}A_{-k}(1_\Omega
)\rangle ,
\end{multline*}
where we used (\ref{bref.7}) in the last step. Applying
(\ref{bref.17}) to $A_{-k}(1_\Omega )$ and the fact that
$B_k(1_\Omega )=\overline{A_k(1_\Omega )}$, we get
\begin{multline}\label{bref.20}
\frac{2}{\pi }\int_\Omega e^{kz-\overline{kz}}AB(1_\Omega
)(z)L(dz)\\
=-\frac{2}{\pi }\langle \frac{1}{2k}e^{-\overline{k\cdot } +k\cdot
}1_\Omega +\overline{r(\cdot ,k)}|e^{-k\cdot +\overline{k\cdot
  }}(-\frac{1}{2\overline{k}}e^{k\cdot -\overline{k\cdot }}1_\Omega
+r(\cdot ,-k))\rangle \\
=\frac{2}{\pi }\frac{1}{4|k|^2}\int_\Omega
e^{kz-\overline{kz}}L(dz)-
\frac{2}{\pi }\int_\Omega \frac{1}{2k}r(z,-k)L(dz)\\+\frac{2}{\pi
}\int_\Omega \overline{r(z,k)}\frac{1}{2\overline{k}}L(dz)
-\frac{2}{\pi }\int_\Omega e^{-kz+\overline{kz}}\overline{r(z,k)}r(z,-k)L(dz).
\end{multline}

\par As we have already seen, the integral in the first term in the
last member is $\mathcal{ O}(|k|^{-3/2})$, so this term is $\mathcal{
  O}(|k|^{-7/2})$. By (\ref{bref.18}) the last term in (\ref{bref.20})
is $\mathcal{ O}(|k|^{-3}\ln |k|)$. Thus (\ref{bref.20}) gives
\begin{equation}\label{bref.21}\begin{split}
&\frac{2}{\pi }\int_\Omega e^{kz-\overline{kz}}AB(1_\Omega
)(z)L(dz)\\
&=-
\frac{2}{\pi }\int_\Omega \frac{1}{2k}r(z,-k)L(dz)+\frac{2}{\pi
}\int_\Omega \overline{r(z,k)}\frac{1}{2\overline{k}}L(dz)
+\mathcal{ O}(|k|^{-3}\ln |k|).
\end{split}
\end{equation}
(\ref{bref.19}) now yields
\begin{equation}\label{bref.22}
\frac{2}{\pi }\int_\Omega e^{kz-\overline{kz}}AB(1_\Omega
)(z)L(dz)=\mathcal{ O}(|k|^{-5/2}).
\end{equation}

\par Before studying the leading asymptotics of the integrals in the
left hand side of(\ref{bref.21}), we shall gain a power of $k$ in the
estimate of the general term in (\ref{bref.3}) for $\nu \ge 2$:
\begin{multline}\label{bref.23}
\frac{2}{\pi }\int_\Omega e^{kz-\overline{kz}}(AB)^\nu (1_\Omega
)(z)L(dz)
=\frac{2}{\pi }\int_\Omega e^{kz-\overline{kz}}A(BA)^{\nu -1}B
(1_\Omega )(z)L(dz)\\
=-\frac{2}{\pi }\langle (BA)^{\nu -1}B(1_\Omega )|e^{-k\cdot
  +\overline{k\cdot }}A_{-k}(1_\Omega )\rangle =\mathcal{ O}(1)|k|^{1-\nu
}|k|^{-1}|k|^{-1}\\=\mathcal{ O}(|k|^{-\nu -1})= \mathcal{ O}(|k|^{-3}),
\end{multline}
since $(BA)^{\nu -1}=\mathcal{ O}(|k|^{1-\nu }):\, L^2(\Omega )\to
L^2(\Omega )$ and $B(1_\Omega ),\, A_{-k}(1_\Omega )=\mathcal{ O}(1/|k|)$
in $L^2(\Omega ).$

\par Combining (\ref{bref.3}), (\ref{bref.22}), (\ref{bref.23}), we
get
\begin{equation}\label{bref.24}
  \overline{R}=\frac{2}{\pi }\int_\Omega e^{kz-\overline{kz}}L(dz)+
  \frac{2}{\pi }\int_\Omega e^{kz-\overline{kz}}AB(1_\Omega
  )L(dz)+\mathcal{ O}(|k|^{-3}),
\end{equation}
and in particular,
\begin{equation}\label{bref.24.5}
  \overline{R}=\frac{2}{\pi }\int_\Omega e^{kz-\overline{kz}}L(dz)+
  \mathcal{ O}(|k|^{-5/2}).
\end{equation}

\par We next study the second term in the right hand side of
(\ref{bref.24}), starting from (\ref{bref.21}). 
By (\ref{bref.16}) and (\ref{bref.25}) we have
\begin{equation}\label{bref.26}
r=\frac{F}{4\pi i\overline{k}}+\frac{\mathcal{O}(1)\ln |k|}{|k|^2} \hbox{
in }L^1(\Omega ),
\end{equation}
where $F=F(z,k)$. Using this in (\ref{bref.21}), we get
\begin{multline}\label{bref.27}
\frac{2}{\pi }\int_\Omega e^{kz-\overline{kz}}AB(1_\Omega
)(z)L(dz)=\\
\frac{2}{\pi }\int_\Omega \frac{1}{2k}\frac{F(z,-k)}{4i\pi
  \overline{k}}L(dz)+
\frac{2}{\pi }\int_\Omega \frac{\overline{F}(z,k)}{-4i\pi
  k}\frac{1}{2\overline{k}}L(dz)+\mathcal{ O}(|k|^{-3}\ln |k| )\\
=\frac{1}{4i\pi ^2|k|^2}\int_\Omega (F(z,-k)-\overline{F}(z,k))L(dz)+
\mathcal{ O}(|k|^{-3}\ln |k| ).
\end{multline}
Here we recall (\ref{bref.10}) for $F(z,k)=F_\Gamma (z)$, where
$\Gamma =\Gamma (k)$ is a deformation of $\partial \Omega $ passing
through the poles $w_+(k)$, $w_-(k)$, situated outside
$\overline{\Omega }$ near the boundary segment $\Gamma _-$ from
$w_+(k)$ to $w_-(k)$ and inside $\Omega $ near the boundary segment
$\Gamma _+$ from $w_-(k)$ to $w_+(k)$ (when following the boundary
with the positive orientation). This choice is given by the method of
steepest descent for $e^{-iu(\cdot ,k)}$. When replacing $k$ with
$-k$, we have $e^{-iu(z,-k)}=e^{iu(z,k)}$ and
$w_{\pm}(-k)=w_{\mp}(k)$. Correspondingly, $\Gamma $ should be replaced
by a contour $\widetilde{\Gamma }$ which is a deformation of $\partial
\Omega $ inwards near the segment $\Gamma _-$ from $w_+(k)$ to
$w_-(k)$ and outwards near the segment $\Gamma _+$ from $w_-(k)$ to
$w_+(k)$.

\par We get
\begin{multline}\label{bref.28}
\int_\Omega F(z,-k)L(dz)=\int_\Omega \int_{\widetilde{\Gamma
  }}\frac{1}{z-w}e^{iu(w,k)}dw L(dz)\\
= \int_{\widetilde{\Gamma
  }} \int_\Omega \frac{1}{z-w}L(dz) e^{iu(w,k)}dw
=-\pi \int_{\widetilde{\Gamma }}D_\Omega (w) e^{iu(w,k)}dw,
\end{multline}
using that $\frac{1}{z-w}e^{iu(w,k)}$ is integrable on $\Omega \times
\widetilde{\Gamma }$ for the measure $L(dz)|dw|$. Here
\begin{equation}\label{bref.29}
D_\Omega (z)=\frac{1}{\pi }\int_\Omega \frac{1}{z-w}L(dw)
\end{equation}
is the solution to the $\overline{\partial }$-problem:
\begin{equation}\label{bref.30}
  \begin{cases}\partial _{\overline{z}}D_\Omega =1_\Omega ,\\
    D_\Omega (z)\to 0,\ z\to \infty .
\end{cases}
\end{equation}
Similarly,
\begin{multline}\label{bref.31}
-\int_\Omega \overline{F(z,k)}L(dz)=\overline{\int_\Omega \int_\Gamma
  \frac{1}{w-z}e^{-iu(w,k)}dwL(dz)}\\
=\overline{\int_\Gamma \int_\Omega
  \frac{1}{w-z}L(dz)e^{-iu(w,k)}dw}
=\pi \overline{\int_\Gamma D_\Omega (w)e^{-iu(w,k)}dw}.
\end{multline}
Using (\ref{bref.28}), (\ref{bref.31}) in (\ref{bref.27}), we get
\begin{multline}\label{bref.32}
\frac{2}{\pi }\int_\Omega e^{kz-\overline{kz}}AB(1_\Omega
)(z)L(dz)=\\
\frac{1}{4i\pi |k|^2}\left( -\int_{\widetilde{\Gamma }}D_\Omega
  (w)e^{iu(w,k)}dw+\overline{\int_\Gamma D_\Omega (w)e^{-iu(w,k)}dw}
\right)+\mathcal{ O}(|k|^{-3}\ln |k| ).
\end{multline}
\begin{remark}
	It is not obvious whether the error term in (\ref{bref.32}) is 
	optimal or a consequence of the applied technique to prove the 
	result. 
\end{remark}

\begin{prop}\label{bref1}
$D_\Omega \in C(\mathbb{C})$ and the restrictions of this function to the
open sets $\Omega $ and $\mathbb{C}\setminus \overline{\Omega }$ extend
to functions in $C^\infty (\overline{\Omega })$ and $C^\infty ({\bf
  C}\setminus \Omega )$ respectively.
\end{prop}
\begin{proof}
The continuity of $D_\Omega $ is clear. We first look at $D_\Omega
(z)$ in $\mathbb{C}\setminus \overline{\Omega }$. Here $D_\Omega (z)$ is
holomorphic and
\begin{equation}\label{bref.33}\begin{split}
\partial _zD_\Omega (z)=&\frac{1}{\pi }\int_\Omega \partial
_z\left(\frac{1}{z-w} \right)L(dw)\\
=&-\frac{1}{\pi }\int_\Omega \partial _w \left(\frac{1}{z-w}
\right)\frac{d\overline{w}\wedge dw}{2i}=\frac{1}{2\pi i}\int_\Omega
d_w\left(\frac{1}{z-w} d\overline{w} \right).
\end{split}
\end{equation}
By Stokes' formula,
\begin{equation}\label{bref.34}
\partial _zD_\Omega (z)=\frac{1}{2\pi i}\int_{\partial \Omega
}\frac{1}{z-w}d\overline{w}=
\frac{1}{2\pi i}\int_0^L\frac{1}{z-\gamma (t)}\overline{\dot{\gamma }(t)}dt
\end{equation}
where $\gamma :[0,L[\ni t\mapsto \gamma (t)\in \partial \Omega $ is a
smooth positively oriented parametrization of $\partial \Omega $.

\par It follows that
\begin{equation}\label{bref.35}
\partial _zD_\Omega (z)=\mathcal{ O}(1)|\ln d(z,\partial \Omega )|,\ z\in
(\mathbb{C}\setminus \overline{\Omega})\cap \mathrm{neigh\,}(\partial
\Omega ),
\end{equation}
where $d(z,\partial \Omega )$ denotes the distance from $z$ to
$\partial \Omega $.

\par For $n=1,2,...$, we apply $\partial _z^n$ to (\ref{bref.34}):
\begin{multline}\label{bref.36}
\partial _z^{n+1}D_\Omega (z)=\frac{1}{2\pi i}\int_{\partial \Omega
}\partial _z^n \frac{1}{z-w}d\overline{w}\\ =
\frac{1}{2\pi i}\int_{\partial \Omega
}(-\partial _w)^n \frac{1}{z-w}d\overline{w}
=\frac{1}{2\pi i}\int_0^L(-\dot{\gamma }^{-1}\partial _t)^n
\left(\frac{1}{z-\gamma (t)} \right)\overline{\dot{\gamma }}dt\\
=\frac{1}{2\pi i}\int_0^L \frac{1}{z-\gamma (t)}\underbrace{(\partial
  _t\circ \dot{\gamma }^{-1})^n\left(\overline{\dot{\gamma
      }}\right)}_{\mathcal{ O}(1)}dt
=\mathcal{ O}(1)|\ln d(z,\partial \Omega )|,
\end{multline}
still for $z\in \mathrm{neigh\,}(\partial \Omega )\setminus
\overline{\Omega }$. Integrating these estimates, we see that
$\partial _z^nD_\Omega (z)$ is bounded on $\mathbb{C}\setminus
\overline{\Omega }$ for $n\in \Omega $ (which we already knew to be valid away
from a neighborhood of $\partial \Omega $) and hence that $D_\Omega
\in C^\infty (\mathbb{C}\setminus \Omega )$.

\par If instead of $D_\Omega $ we look at
\begin{equation}\label{bref.37}
D_{\Omega ,\psi }(z)=\frac{1}{\pi }\int_\Omega \frac{1}{z-w}\psi (w)L(dw)
\end{equation}
for some $\psi \in C^\infty (\overline{\Omega })$ with $\psi =1$ near
$\partial \Omega $, we still have $D_{\Omega ,\psi }\in C^\infty (
\mathbb{C}\setminus \Omega )$. Indeed, we get very much as in
(\ref{bref.33}),
\begin{multline*}
\partial _zD_{\Omega ,\psi }(z)=\frac{1}{\pi }\int_\Omega \partial
_z\left(\frac{1}{z-w} \right) \psi (w)L(dw)\\
=-\frac{1}{\pi }\iint_\Omega \partial _w\left(\frac{1}{z-w} \right)
\psi (w) \frac{d\overline{w}\wedge dw}{2i}\\
=\frac{1}{2\pi i}\int_\Omega d_w \left(\frac{\psi (w)}{z-w}d\overline{w}
\right)+\frac{1}{\pi }\iint_\Omega \frac{\partial _w\psi }{z-w}
\frac{d\overline{w}\wedge dw}{2i},
\end{multline*}
where the last integral belongs to $C^\infty (\mathbb{C}\setminus \Omega
)$ and the second last integral is equal to 
$\frac{1}{2\pi i}\int_{\partial\Omega}^{}\frac{1}{z-w}d\bar{w}$ and is also in $C^\infty (\mathbb{C}\setminus \Omega
)$ as we have just seen.

\par We finally show that $D_\Omega (z)$, $z\in \Omega $, extends to a
function in $C^\infty (\overline{\Omega })$. Let $\chi \in C_0^\infty
(\mathbb{C})$ be equal to 1 near $\overline{\Omega }$. Then
$f(z)=\frac{1}{\pi }\int_{\mathbb{C}}(z-w)^{-1}\chi (w)L(dw)$ solves
$\partial _{\overline{z}}f=\chi $ and belongs to $C^\infty (\mathbb{C})$. For $z\in \Omega $ we have
$$
D_\Omega (z)=f(z)-\frac{1}{\pi }\int_{\mathbb{C}\setminus \Omega
}\frac{1}{z-w}\chi (w) L(dw).
$$

\par As in the remark above about $D_{\Omega ,\psi }$ in $\mathbb{C}\setminus \overline{\Omega }$ we see that $\frac{1}{\pi }\int_{\mathbb{C}\setminus \Omega
}\frac{1}{z-w}\chi (w) L(dw)$ belongs to $C^\infty (\overline{\Omega
})$. Hence $D_\Omega (z)$ extends from $\Omega $ to a smooth function
on $\overline{\Omega }$.
\end{proof}

This completes the proof of Theorem \ref{maintheorem} in the case 
when $\partial\Omega$ is analytic. 

\par The exponentials in the integrals in (\ref{bref.32}) behave like
Gaussians peaked at $w_\pm (k)$, and hence the integrals are
$\mathcal{ O}(|k|^{-1/2})$ and the contribution from outside any fixed
neighborhood of $\{w_+(k),w_-(k) \}$ is exponentially small. The
proposition implies that $D_\Omega (w)$ is a Lipschitz function and
therefore we modify the integrals by $\mathcal{ O}(|k|^{-1})$ only, if we
replace $D_{\Omega }(w)$ in a neighborhood of $w_\pm (k)$ by $D_\Omega
(w_\pm (k))$. Thus for instance
\begin{multline}\int_{\widetilde{\Gamma }}D_\Omega (w)e^{iu(w,k)}dw=
  D_\Omega
(w_+(k))
\int_{\widetilde{\Gamma }\cap \mathrm{neigh\,}(w_+(k))}e^{iu(w,k)}dw\\
+
D_\Omega
(w_-(k))
\int_{\widetilde{\Gamma }\cap
  \mathrm{neigh\,}(w_-(k))}e^{iu(w,k)}dw+\mathcal{ O}(k^{-1}).
  \label{Do3}
\end{multline}

\begin{remark}\label{remark}
	We now drop the analyticity 
assumption and assume that $\Omega\Subset \mathbb{C}$ is open with 
smooth boundary and strictly convex. We have used the analyticity 
assumption in (\ref{bref.9}), (\ref{bref.10}), where $iu$ is the 
holomorphic extension to a neighbourhood of $\partial \Omega$ of the 
function $kw-\overline{kw}$, $w\in \partial\Omega$. In the merely 
smooth case we let $iu(w,k)$ denote an almost holomorphic extension 
of $\partial \Omega\ni w\mapsto kw-\overline{kw}$ and define $F$ as 
in (\ref{bref.10}), using the modified function $u$. Stokes' formula 
now produces a small error term to be added to (\ref{bref.10}) and as in 
(2.18) in \cite{KlSjSt20b} and the subsequent discussion, we get 
$$f(z,k)=\frac{1}{2i\bar{k}}F(z)+\frac{\pi}{\bar{k}}\left(
e^{-iu(z,k)}(1_{\Omega_{-}}(z)-1_{\Omega_{+}}(z))+
e^{-kz+\overline{kz}}1_{\Omega}(z)\right)$$
$$+\mathcal{O}(\langle z\rangle 
^{-1}|k|^{-\infty}).$$
The discussion after (\ref{bref.10}) goes through with only a minor 
change: In the formula (\ref{bref.25}) for $r$ we have to add a 
remainder $\mathcal{O}(|k|^{-\infty})$. But this does not affect the 
subsequent estimates, and we get Theorem \ref{maintheorem} also in 
the more general case of a smooth boundary. 
\end{remark}

\section{Stationary phase approximation}
\setcounter{equation}{0}

To compute the leading orders in $1/|k|$ of the reflection 
coefficient, we apply a standard stationary phase approximation. 
Since higher order terms in this approximation are needed here, we briefly 
summarize some facts on the approach. 

\subsection{Two term stationary phase expansion}
We have in mind $\int_{\partial\Omega} e^{-kz + \overline{kz}}$ after 
a change of contour from $\partial \Omega$ to $\Gamma$ and choosing a parametrisation. Let $\varphi\in C^\infty (\mbox{neigh} (0, \mathbb{R}))$ satisfy  $\varphi(0) = 0$, $\varphi'(0) = 0$, $\Re\varphi''(0)>0$, $a\in C^\infty(\mbox{neigh}(0, \mathbb{R}))$. Consider 
\begin{equation*}
I(\varphi, a; h) = h ^{-\frac{1}{2}}\int_Ve^{-\varphi(t)/h}a(t)dt, \quad V = \mbox{neigh} (0, \mathbb{R}).  
\end{equation*}
We already know that $I \sim I_0 + I_1h + \dots,$ where $I_0 =\frac{\sqrt{2\pi}}{\sqrt{\varphi''(0)}} $ with the natural choice of the branch of the square root and our problem is to compute $I_1h$. 
Write $\varphi(t) = \varphi_2t^2 + \psi(t)$, 
 where $\psi(t) = \mathcal{O}(t^3)$. Put
\begin{equation*}
\varphi(t;s) = \varphi_2t^2+ s\psi, \quad 0\leq s\leq 1.
\end{equation*}
Clearly $I(\varphi, a; h)$ is smooth in $s$ with $\partial^k_sI(\varphi_s, a ; h) = \mathcal{O}(h^{k/2})$, $k = 0, 1, 2, \dots$, hence by a limited Taylor expansion 
\begin{align}
I(\varphi, a; h) &= I(\varphi, a ; h)|_{s=0} + (\partial_s) I(\varphi, a;h)|_{s=0} 
+ \frac{1}{2}(\partial^2_s) I(\varphi, a;h)|_{s=0} + \mathcal{O}(h^{3/2})\nonumber \\
&=I(\varphi_2t^2, a ; h) - I(\varphi_2t^2, a\psi/h ; h) + \frac{1}{2}I(\varphi_2t^2, 
a(\psi/h)^2 ; h) + \mathcal{O}(h^{3/2}).
\label{expandI}
\end{align}
Indeed
\begin{equation}
I(\varphi, \mathcal{O}(t^n); h) = \mathcal{O}(h^{\frac{n}{2}}).
\label{estI}
\end{equation} 
Using again (\ref{estI}) we can replace $a$, $a\frac{\Psi}{h}$ and $a(\frac{\Psi}{h})^2$ by their limited Taylor sums modulo $\mathcal{O}(t^3)$, $\mathcal{O}(\frac{t^5}{h})$ and $\mathcal{O}(\frac{t^{7}}{h^2})$ respectively. \\
Writing $\psi\sim \varphi_3t^3+ \varphi_4t^4 + \dots$, $a\sim a_0+a_1t +a_2t^2 + \dots$ and using also that $I(\varphi_2t^2, t^n; h) = 0$ when $n$ is odd (up to an exponentially small error if $V$ is not symmetric around $t = 0$), we get 
\begin{equation*}
\begin{split}
I(\varphi_2t^2, a; h) &= I (\varphi_2t^2,a_0, h ) +  I (\varphi_2t^2,a_2t^2, h ) + \mathcal{O}(h^2) \\
 &=a_0 I(\varphi_2t^2, 1 ;1) + h a_2I(\varphi_2t^2, t^2;1) + \mathcal{O}(h^2), 
\end{split}
\end{equation*}
and 
\begin{equation*}
\begin{split}
I(\varphi_2t^2, \frac{a \psi}{h}; h) &= I (\varphi_2t^2,(a_0+a_1t)\frac{\varphi_3t^3 + \varphi_4t^4}{h};h) + \mathcal{O}(h^{3/2}) \\
 &=I(\varphi_2t^2, \frac{(a_0\varphi_4+a_1\varphi_3)t^4}{h} ;h) + 
 \mathcal{O}(h^{3/2}) \\
 &=   h(a_0\varphi_4+ a_1\varphi_3) I(\varphi_2t^2, t^4;1)) + 
 \mathcal{O}(h^{3/2}) ,
\end{split}
\end{equation*}
as well as 
\begin{equation*}
\begin{split}
I(\varphi_2t^2, a(\frac{\psi}{h})^2; h) &= I (\varphi_2t^2,\frac{a_0\varphi_3^2t^6}{h^2}, h ) +   \mathcal{O}(h^{3/2}) \\
&= ha_0 \varphi_3^2I(\varphi_2t^2, t^6 ;1) + \mathcal{O}(h^{3/2}). 
\end{split}
\end{equation*}
We now have integrals with quadratic exponent and polynomial amplitudes and up to exponentially small corrections from now on, we integrate over $\mathbb{R}$ instead of $V$.  

\subsection{Reduction to the case of an exact quadratic} 
Now we consider the reduction to the case of a quadratic exponential, 
\begin{equation*}
I(\varphi_0, t^{2n}; 1) = \int e^{-\varphi_2t^2}t^{2n} dt.  
\end{equation*}
Reparametrise  $t^2 = \frac{\tau^2}{2\varphi_2}$, $t = \frac{1}{(2\varphi_2)^{1/2}}\tau$, and $dt = \frac{1}{\sqrt{2\varphi_2}}d\tau = \frac{1}{\sqrt{\varphi''(0)}}d\tau$ 
\begin{equation*}
I(\varphi_2t^2,t^{2n};1) = 
(2\varphi_2)^{-n-\frac{1}{2}}I(\frac{\tau^2}{2};\tau^{2n}; 1)  = 
(\varphi''(0))^{-n-\frac{1}{2}}I(\frac{\tau^2}{2}, \tau^{2n}; 1).
\end{equation*}

Integrate by parts when $n\geq 1$:
\begin{equation*}
\begin{split}
I(\tau^2/2, \tau^{2n}; 1)  &= 
\int_{\mathbb{R}}e^{-\tau^2/2}\tau^{2n}d\tau = - \int_{\mathbb{R}} \partial_\tau(e^{-\tau^2/2})\tau^{2n-1}d\tau \\
 &=\int e^{-\tau^2/2}(2n-1)\tau^{2(n-1)}d\tau = (2n-1)I(\tau^2/2, 
 \tau^{2(n-1)}; 1). 
\end{split}
\end{equation*}
In particular 
\begin{equation*}
\begin{split}
I(\tau^2/2,\tau^2; 1) &= I(\tau^2/2, 1; 1) = \sqrt{2\pi},\\
I(\tau^2/2,\tau^4; 1) &= 3I(\tau^2/2, 1; 1) = 3\sqrt{2\pi},\\
I(\tau^2/2,\tau^6; 1) &= 15I(\tau^2/2, 1; 1) = 15\sqrt{2\pi}.
\end{split}
\end{equation*}
We combine the different identities: 
\begin{equation}
\begin{split}
I(\varphi, a; h) &= I(\varphi_2t^2, a; h) - I(\varphi_2t^2, a\frac{\psi}{h}; h) + \frac{1}{2}I(\varphi_2t^2, a(\frac{\psi}{h})^2; h) + \mathcal{O}(h^{\frac{3}{2}})  \\
&=a_0I(\varphi_2t^2, 1; 1) + ha_2I(\varphi_2t^2, t^2; 1) -h(a_0\varphi_4 + 
a_1\varphi_3)I(\varphi_2t^2, t^4; 1)\\
& \quad \quad  + \frac{1}{2}h a_0\varphi_3^2 I(\varphi_2t^2, t^6; 1) + \mathcal{O}(h^{\frac{3}{2}})  \\
 &= \sqrt{2\pi}a(0)\varphi''(0)^{-\frac{1}{2}}+ 
 \sqrt{2\pi}h\left[\frac{1}{2}a''(0)\varphi''(0)^{-\frac{3}{2}} \right.\\
&\quad\quad - \left(\frac{1}{8}a(0)\varphi^{(4)}(0) + \frac{1}{2}a'(0)\varphi^{(3)}(0)\right)\varphi''(0)^{-\frac{5}{2}} \\
& \quad\quad\quad\quad \left. + \frac{5}{24} a(0)\varphi^{(3)}(0)^2\varphi''(0)^{- \frac{7}{2}}\right] + \mathcal{O}(h^{\frac{3}{2}}) .
\end{split}
\label{I0}
\end{equation}
Here we recall that $\varphi_3 = \varphi^{(3)}(0)/6$ and $\varphi_4 = 
\varphi^{(4)}(0)/24$, $a_0 = a(0)$, $a_1 = a'(0)$, $a_2 = \frac{1}{2}a''(0)$. We know that $I(\varphi, a ; h)$ has an asymptotic expansion in integer 
powers of $h$, so the remainder can be improved to 
$\mathcal{O}(h^2)$. In view of the application to 
$\int_{\partial\Omega} e^{-kz + \overline{kz}}dz$, we try to express 
the result in terms of $\varphi/h=:\Phi$, and it then seems 
convenient to replace $I(a, \varphi; h)$ above by $J(a, 
\Phi):=h^{\frac{1}{2}}I(a, \varphi; h) = \int_Ve^{-\Phi(t)}a(t)dt$. 
From (\ref{I0}) we get using $\varphi = h\Phi$, 
\begin{align}
\nonumber \frac{1}{\sqrt{2\pi}}J(\varphi, a)& =   a(0)\Phi''(0)^{-\frac{1}{2}}  + \frac{a''(0)}{2}\Phi''(0)^{-\frac{3}{2}} \\ 
\nonumber& - \left( a(0)\frac{\Phi^{(4)}(0)}{8} + a'(0)\frac{\Phi^{(3)}(0)}{2} \right)\Phi''(0)^{-\frac{5}{2}}   \\
&+  \frac{5}{24}a(0)\Phi^{(3)}(0)^2\Phi''(0)^{-\frac{7}{2}} 
 + \mathcal{O}(h^{2}).
%
%
\label{defJ}
\end{align}
Notice that the first term in the final expression is homogeneous of degree $-\frac{1}{2}$ in $\Phi$, while the following one is homogeneous of degree $-\frac{3}{2}$.

\subsection{Stationary phase approximation for the reflection coefficient}
We now apply the above results to the reflection coefficient (we only discuss the analytic case here, see Remark 
\ref{remark} for a generalization to the smooth  case). 
We also assume in the following that $\Phi$ does not vanish at the 
stationary points.


In application to $\int_{\partial\Omega}e^{kz-\overline{kz}}$, let 
$[0, L[\ni t \mapsto \gamma(t)$ parametrize the boundary, so 
$\int_{\partial\Omega}e^{kz-\overline{kz}}dz = 
\int_0^Le^{-\Phi(t)}a(t)dt$, with 
\begin{equation}
	\Phi(t) = -(k\gamma(t) - 
\overline{k\gamma(t)}),\quad a(t) = \dot{\gamma}(t).
	\label{Phidef}
\end{equation}
Apply (\ref{defJ}) 
with $\mathcal{O}(h^2) = \mathcal{O}(|k|^{-2})$. 

\begin{remark}\label{rem2}
It remains to choose 
the correct branches of $(\Phi''(t))^{\frac{1}{2}}$ at $t=t_{\pm}$. 
We adapt the notation of \cite{KlSjSt20}, 
$iu(z,k)=kz-\overline{kz}$, $z\in\partial \Omega$. We have 
$\Phi(t)=-iu(\gamma(t),k))=:-iU(t,k)$ which is purely imaginary with 
two non degenerate critical points at $t=t_{+},t_{-}$ corresponding to 
the poles $w_{+},w_{-}$ respectively. By contour deformation we see that the 
stationary phase approximation is still valid with the branch of 
$(\Phi''(t_{\pm}))^{1/2}$ obtained as the limit of 
$(F_{\pm,\epsilon})^{1/2}$, where $F_{\pm,\epsilon}$ is a sequence 
converging to $\Phi''(t_{\pm})$ when $\epsilon\searrow 0$ with the 
property $\Re F_{\pm,\epsilon}>0$. We get 
$(\Phi''(t_{\pm}))^{1/2}=e^{\mp i\pi/4}|U''(t_{\pm},k)|^{1/2}$.
\end{remark}

We get for  (\ref{Do3}) 
\begin{equation}
	\int_{\widetilde{\Gamma }}D_\Omega (w)e^{iu(w,k)}dw
	= 
	\sqrt{2\pi} \sum_{t=t_{+},t_{-}}^{}e^{-\Phi(t)}D_{\Omega}(\gamma(t))
	a(t)(\Phi''(t))^{-\frac{1}{2}}
+\mathcal{O}(|k|^{-1}),
		\label{stationary}
\end{equation}
where $\Phi$ and $a$ are defined in (\ref{Phidef}), and where the 
signs of the roots are chosen as detailed in Remark \ref{rem2}.

The leading term in the reflection coefficient (\ref{eqmt}) 
is due to the term 
$\int_{\Omega}^{}e^{kz-\overline{kz}}L(dz)$. We apply Stokes' 
theorem as before to write this in the form of an integral over 
$\partial \Omega$ and apply a stationary phase approximation, 
\begin{equation}
	\begin{split}
	\frac{i}{2\bar{k}}\int_{\partial\Omega}^{}e^{kz-\overline{kz}}dz
&= 
	\frac{i\sqrt{2\pi }}{2\bar{k}}\sum_{t=t_{+},t_{-}}^{}
	 e^{-\Phi(t)}\left(a(t)(\Phi''(t))^{-\frac{1}{2}} + 
	 \frac{a''(t)}{2}(\Phi''(t))^{-\frac{3}{2}}\right.\\
	 &- \left( \frac{a(t)}{8} 
	 \Phi^{(4)}(t) + 
	 \frac{1}{2}a'(t){\Phi^{(3)}(t)} 
	 \right)(\Phi''(t))^{-\frac{5}{2}}\\
	 &\quad \quad\quad \left.+ 
	 \frac{5}{24}a(t) \Phi^{(3)}(t)^2(\Phi''(t))^{-\frac{7}{2}} \right) 
+\mathcal{O}(|k|^{-3}).
\end{split}
	\label{stationary2}
\end{equation}
Again $\Phi$ and $a$ are defined in (\ref{Phidef}), and the sign of 
the roots are chosen as explained in Remark \ref{rem2}. With 
(\ref{stationary}) and (\ref{stationary2}) we get for the reflection 
coefficient (\ref{eqmt}) relation (\ref{eqmt2}). 

\subsection{Example: Characteristic function of the unit disk}

In general, we cannot compute explicitly $D_\Omega (z)$, but in
the special case of the unit disc, we have
$$
D_{D(0,1)}(z)=\begin{cases}\overline{z},\ |z|\le 1,\\
1/z,\ |z|\ge 1.
\end{cases}
$$

For the stationary phase approximation, we parametrize $\partial 
\Omega$ via $\gamma(t)=e^{it}$, $t\in [0,2\pi[$. Writing 
$k=|k|e^{i\vartheta}$, $\vartheta\in \mathbb{R}$, we have 
$\Phi=-2i|k|\sin(t+\vartheta)$ and $a=ie^{it}$ in (\ref{Phidef}). The 
critical points are $t^{\pm}=\pm \frac{\pi}{2}-\vartheta$ (the 
relation to the previously introduced $t_{\pm}$ is $t^{\pm}:=t_{\mp}$). We have 
$\Phi(t^{\pm})=-\Phi''(t^{\pm})=\Phi^{(4)}(t^{\pm})=\mp2i|k|$, 
whereas $\Phi'(t^{\pm})=\Phi^{(3)}(t^{\pm})=0$ and 
$a(t^{\pm})=-a''(t^{\pm})=\mp 
e^{-i\vartheta}$. 
This implies for the right hand 
side of (\ref{stationary2} 
\begin{equation} 
	\frac{i}{2\bar{k}}\int_{\partial\Omega}^{}e^{kz-\overline{kz}}dz 
	=\frac{\sqrt{\pi}}{|k|^{3/2}}\left(\sin(2|k|-\pi/4)+\frac{3}{16|k|}\cos(2|k|-\pi/4)\right)+\mathcal{O}(|k|^{-7/2}).
	\label{D2}
\end{equation}

Similarly we get for (\ref{stationary}) 
\begin{equation}
		\int_{\widetilde{\Gamma }}D_\Omega (w)e^{iu(w,k)}dw
	= \sqrt{\frac{\pi}{|k|}}2i\cos(2|k|-\pi/4)
+\mathcal{O}(|k|^{-1})
	\label{D1}.
\end{equation} 
Thus we get for the leading terms of the reflection coefficient 
(\ref{eqmt2}) the 
result conjectured in \cite{KlSjSt20} (note that the formula for 
$R/2$ was given there),
\begin{equation}
	R \approx \frac{2}{\sqrt{\pi 
	|k|^{3}}}\left(\sin(2|k|-\pi/4)-\frac{5}{16|k|}\cos(2|k|-\pi/4)\right)
	\label{Rasym}.
\end{equation}
Note that the conjectured error term is smaller than what is proven 
in this paper. 

\section{Conclusion}
In this paper, we have presented asymptotic relations for large $|k|$ 
for the solutions to the Dirac system (\ref{dbarphi}) subject to the 
asymptotic conditions (\ref{Phisasym}). Previous results for 
potentials being the characteristic function of a compact domain with 
smooth convex boundary have been improved and extended to the 
reflection coefficient, the scattering data in the context of an 
integrable systems approach to the DS II equation. The results are 
now extended to  $\mathcal{O}(|k|^{-5/2})$ which makes it possible  
to apply these 
formulae to complement numerical computations in order to get the 
reflection coefficient for all $k\in\mathbb{C}$ with the same 
precision as discussed in \cite{KlSjSt20}. This allows to treat the 
reflection coefficient with a \emph{hybrid 
approach} combining numerical and analytical results.

An interesting question in the context of EIT would be to extend the 
results of this paper to a compact domain with cavities. Since in 
applications to the human body, the organs of a patient are of 
essentially constant conductivity, this corresponds to a situations 
of a domain with compact support and cavities all of which have 
smooth compact boundaries. The boundary data at the cavities are a 
consequence of the conductivity in the interior. It will be the subject of further work to 
adapt the present formulae to this case. An interesting question to 
be addressed is also to find the optimal error term in Theorem 
\ref{maintheorem}.

 \end{document}